\documentclass[final]{siamltex}

\usepackage[dvips]{graphicx}
\usepackage{amsmath}
\usepackage{amsfonts}

\newtheorem{remark}{{\it Remark}}[section]

\newcommand{\tra}[1]{\,{\vphantom{#1}}^{\textsc{t}}\!{#1}}

\begin{document}

\title{Discrete Calculus of Variations for Quadratic Lagrangians. Convergence Issues}

\author{Philippe Ryckelynck\footnotemark[1]~~ Laurent Smoch\footnotemark[1]} 

\thanks{{ULCO, LMPA, F-62100 Calais, France.\hspace{1cm}e-mail: \{ryckelyn,smoch\}@lmpa.univ-littoral.fr\\
\indent \indent Univ Lille Nord de France, F-59000 Lille, France. CNRS, FR 2956, France.}}
\maketitle
\baselineskip15pt

\begin{abstract}
We study in this paper the continuous and discrete Euler-Lagrange equations arising from a quadratic lagrangian. Those equations may be thought as numerical schemes and may be solved through a matrix based framework. When the lagrangian is time-independent, we can solve both continuous and discrete Euler-Lagrange equations under convenient oscillatory and non-resonance properties. The convergence of the solutions is also investigated. In the simplest case of the harmonic oscillator, unconditional convergence does not hold, we give results and experiments in this direction.
\end{abstract}



\section{Introduction}

\medskip 

The principle of least action may be extended to the case of non-differentiable dynamical variables by replacing in the lagrangian $\mathcal{L}(\mathbf{x},\dot{\mathbf{x}})$ the derivative $\mathbf{\dot x}(t)$ of the dynamical variable $\mathbf{x}(t)$ with a $2N+1$-terms scale derivative 
\begin{equation}
\Box _\varepsilon \mathbf{x}(t)=\sum_{i=-N}^Nc_i\mathbf{x}(t+i\varepsilon)\chi _{-i}(t),~~t\in[a,b],
\label{ourbox}
\end{equation}
see \cite{Cre,CFT,RS2}. Here, $\varepsilon$ stands for some time delay and $\chi _i(t)$ denotes the characteristic function of the interval $[\max(a,a+i\varepsilon ),\min (b,b+i\varepsilon )]$. Critical points of classical actions are characterized by the classical Euler-Lagrange equations $\nabla_\mathbf{x}\mathcal{L}-d/dt\nabla_{\dot{\mathbf{x}}}\mathcal{L}=0$. Similarly, we proved in \cite{RS2} that the equations of motion for discretized actions are
\begin{equation}
\nabla _\mathbf{x} \displaystyle \mathcal{L}+\Box_{-\varepsilon }\nabla_{\mathbf{\dot x}}\mathcal{L}=0.
\label{intrinsic}
\end{equation}
We abbreviate as C.E.L. and D.E.L. the classical and discrete Euler-Lagrange systems of equations respectively.

In this paper we work with lagrangians of the shape $\mathcal{L}(\mathbf{x},\dot{\mathbf{x}})$ and $\mathcal{L}(\mathbf{x},\Box_\varepsilon\mathbf{x})$ where $\mathcal{L}:\mathbb{C}^d\times\mathbb{C}^d\rightarrow \mathbb{C}$ is a quadratic polynomial. We are interested in solving C.E.L. and D.E.L. under Dirichlet conditions. More accurately, we study the existence and the unicity of pseudo-periodic solutions $\mathbf{z}(t)$ of C.E.L. and $\mathbf{y}_\varepsilon(t)$ of D.E.L., $\varepsilon$ being fixed. The underlying assumptions for this to occur may be thought as an ``oscillatory" condition for the lagrangian $\mathcal{L}(\mathbf{x},\mathbf{y})$ and as a ``non-resonance" condition for the Dirichlet problem associated to C.E.L. and D.E.L.. With this in mind, we address the problem of convergence of $\mathbf{y}_\varepsilon(t)$ to $\mathbf{z}(t)$.

The paper is organized as follows. Section 2 gives notation and basic definitions used throughout. 
In Section 3, we develop a matricial based framework to solve D.E.L. for all quadratic time-dependent lagrangian.
In Section 4, we provide under mild assumptions formulas for the components of the pseudo-periodic solutions of C.E.L. and D.E.L. when the lagrangian does not explicitly depend on time. This allows us to compute in some particular cases the phases of $\mathbf{y}_\varepsilon(t)$ and $\mathbf{z}(t)$ help to the matrix $A$. 
Section 5 is a preliminary discussion of convergence of $\mathbf{y}_\varepsilon(t)$, uniformly locally in $]a,b[$, as $\varepsilon$ tends 0 for stationary lagrangians and pseudo-periodic solutions. 
If $\mathcal{L}$ is a non-resonant oscillatory lagrangian and $\Box_\varepsilon$ is a well-chosen three-terms operator, the previous convergence property is the content of our main theorem which is proved in Section 6. 
In Section 7, we give numerical experiments to illustrate the non-unconditional convergence of solutions.

\section{Preliminaries}

First, let us collect some notation and definitions from \cite{RS2}. If the context is clear enough, $i=\sqrt{-1}$. Let $[a,b]$ be some interval of time and a time delay $\varepsilon >0$ be fixed throughout. The integers $d$ and $N$ denote respectively the ``physical" dimension and the number of samples in $\mathbb{C}^d$. We define for $t_0\in[a,b]$ the grid $\mathcal{G}_{t_0,\varepsilon}=\{t_0+n\varepsilon ,n\in\mathbb{N}\}\cap[a,b]$.  We denote by $I_d$ the identity matrix of size $d$.

Let $\mathcal{C}_{pw}(d,N)$ be the space of the functions $\mathbf{x}:[a,b]\rightarrow \mathbb{C}^d$ continuous on each interval $[a+i\varepsilon,a+(i+1)\varepsilon ]\cap [a,b]$ for all $i\in\{-N,\ldots,N\}$. The two functional spaces $C^1([a,b],\mathbb{C}^d)$ and $\mathcal{C}_{pw}(d,N)$ are Banach algebras with uniform norms. The operator $\Box_\varepsilon$ given in (\ref{ourbox}) is a continuous linear  endomorphism of $\mathcal{C}_{pw}$.

Now, let be given six mappings $P,Q,J_1\in C^1([a,b],\mathbb{C}^{d\times d})$, $J_2,J_3\in C^1([a,b],\mathbb{C}^d)$ and $J_4\in C^1([a,b],\mathbb{C})$.
We suppose that for all $t$, $P(t)$ and $Q(t)$ are symmetric and $J_1(t)$ is skew-symmetric. We set
\begin{equation}
\mathcal{L}(\mathbf{x},\mathbf{y})=\frac 12\tra{\mathbf{y}}P\mathbf{y}+\frac12\tra{\mathbf{x}}Q\mathbf{x}+\tra{\mathbf{x}}J_1\mathbf{y}+\tra{J_2}\mathbf{y}+\tra{J_3}\mathbf{x}+J_4.  
\label{bilineaire}
\end{equation} 
and we define the quadratic lagrangians $\mathcal{L}(\mathbf{x},\dot{\mathbf{x}})$ and $\mathcal{L}(\mathbf{x},\Box_\varepsilon\mathbf{x})$. If the coefficients in (\ref{bilineaire}) do not depend explicitly on time, we shall say that $\mathcal{L}$ is stationary.

We will consider actions $\mathcal{A}_{cont}(\mathbf{x})$ and $\mathcal{A}_{disc}(\mathbf{x})$ of the shape 
\begin{equation}
\mathcal{A}_{cont}(\mathbf{x})=\int_a^b \mathcal{L}(\mathbf{x},\mathbf{\dot{x}})(t)dt,\hspace{1cm}\mathcal{A}_{disc}(\mathbf{x})=\int_a^b \mathcal{L}(\mathbf{x},\Box_\varepsilon{\mathbf{x}})(t)dt.
\label{action}
\end{equation}
The actions $\mathcal{A}_{cont}:C^1([a,b],\mathbb{C}^d)\rightarrow\mathbb{C}$ and $\mathcal{A}_{disc}:\mathcal{C}_{pw}(d,N)\rightarrow \mathbb{C}$ are continuous and Fr\'echet differentiable everywhere.

We give in \cite{RS2} the necessary first order conditions of local optimum of $\mathcal{A}_{cont}$ and $\mathcal{A}_{disc}$ under the Dirichlet constraints $\mathbf{x}(a)=\mathbf{d}_a$ and $\mathbf{x}(b)=\mathbf{d}_b$ in the previous spaces, where $\mathbf{d}_a$ and $\mathbf{d}_b$ are two fixed vectors in $\mathbb{C}^d$. The Euler-Lagrange equations associated to each action in (\ref{action}) can be written as 
\begin{eqnarray}
-P\ddot {\mathbf{x}}+(-\dot P+2J_1)\dot {\mathbf{x}}+(\dot J_1+Q)\mathbf{x}-\dot{J_2}+J_3=0,  
\label{elho}\\
\Box_{-\varepsilon }(P\Box _\varepsilon \mathbf{x})-\Box _{-\varepsilon }(J_1\mathbf{x})+J_1\Box _\varepsilon \mathbf{x}+Q\mathbf{x}+\Box _{-\varepsilon
}J_2+J_3=0.  
\label{mchomat}
\end{eqnarray}

The problem of convergence as $\varepsilon$ tends to $0$ of the operator in the l.h.s. of (\ref{mchomat}) to the corresponding operator in the l.h.s. of (\ref{elho}) has been studied in \cite{RS2}. In this context, we introduced the class of discretization operators given by
\begin{equation}
\Box^{[r,s]}_\varepsilon \mathbf{x}(t)=-\chi_{+1}(t)\frac{s}{\varepsilon}\mathbf{x}(t-\varepsilon)+\frac{s-r}{\varepsilon}\mathbf{x}(t)+\chi_{-1}(t)\frac{r}{\varepsilon}\mathbf{x}(t+\varepsilon).
\label{ourboxpq}
\end{equation}
where $r,s\in\mathbb{C}$. In fact, (\ref{ourboxpq}) gives the shape of three-terms operators satisfying $\Box_\varepsilon(1)(t)=0$ inside $[a+2\varepsilon,b-2\varepsilon]$. 

Without assuming the convergence of the schemes in the previous sense, we focus on the following two problems. Are the Dirichlet problems for (\ref{elho}) and (\ref{mchomat}) well-posed? Are there periodic or pseudo-periodic solutions? In fact, if $\varepsilon=(b-a)/M$ for some $M\in\mathbb{N}^\star$ and if $\mathbf{x}_\varepsilon(t)$ is a solution of D.E.L., then $\mathbf{x}_\varepsilon$ is uniquely determined on the grid $\mathcal{G}_{a,\varepsilon}$. We shall see later how to construct from $\mathbf{x}_\varepsilon$ the unique corresponding pseudo-periodic solution $\mathbf{y}_\varepsilon$ of (\ref{mchomat}).

\section{An effective method for solving D.E.L.}

In this section, the datas $N,d,\mathcal{L}$, $\mathbf{d}_a,\mathbf{d}_b$, $\varepsilon,\Box_\varepsilon$ are fixed but arbitrary.

\subsection{D.E.L. as delayed functional equations}

The equations (\ref{mchomat}) may be thought as a mixture between recurrence equations and delayed functional equations. Let us transform the problem of solving (\ref{mchomat}) into an infinite set of problems, each of one dealing with recurrence vector equations with the additional difficulty of the perturbation of the boundaries. We can identify each function $\mathbf{x}:[a,b]\rightarrow \mathbb{C}^d$ to the
infinite set of finite sequences $(\mathbf{x}(t_0+n\varepsilon ),n\in\mathbb{Z})$ with indices such that $\frac{a-t_0}\varepsilon \leq n\leq 
\frac{b-t_0}\varepsilon $ and where $t_0$ lies in an interval of length $\varepsilon $. In fact, because (\ref{mchomat}) involves the second order operator $\Box_{-\varepsilon}\Box_\varepsilon$, it may be formulated in an abstract manner as 
\begin{equation}
\mathbf{F}(t,\mathbf{x}(t-2N\varepsilon ),\mathbf{x}(t-(2N-1)\varepsilon
),\ldots ,\mathbf{x}(t+(2N-1)\varepsilon ),\mathbf{x}(t+2N\varepsilon ))=0
\label{formulerec}
\end{equation}
where $\mathbf{F}:\mathbb{C}^{1+(4N+1)d}\rightarrow \mathbb{C}^d$ contains the coordinates of the l.h.s. of (\ref{mchomat}). Hence, we solve (\ref{formulerec}) with respect to $\mathbf{x}(t_0+2N\varepsilon )$ for fixed $t_0$, or what amounts to the same thing by expressing $\mathbf{x}(t_0+4N\varepsilon)$ as a function of $\mathbf{x}(t_0+k\varepsilon)$ for $k\in \{0,\ldots ,4N-1\}$. 

For instance, the case $N=1$ and $d$ arbitrary is the most interesting one, and we may rewrite in this case the equations (\ref{mchomat}) as a system of $d$ equations 
\begin{equation}
{\small \begin{array}{ll}
\phantom{+}
x_i(t-2\varepsilon ) & [c_1c_{-1}\chi _1\chi _2P_{ij}(t-\varepsilon)]\\ 
+x_i(t-\varepsilon ) & [c_0c_1\chi_1P_{ij}(t-\varepsilon)+c_0c_{-1}\chi _0\chi_1P_{ij}(t)+c_1\chi_1(J_1)_{ji}(t-\varepsilon)+c_{-1}\chi_1(J_1)_{ij}(t)]\\ 
+x_i(t) & [c_1^2\chi _0\chi _1P_{ij}(t-\varepsilon)+c_0^2\chi _0P_{ij}(t)+c_{-1}^2\chi_0\chi _{-1}P_{ij}(t+\varepsilon )+\chi _0Q_{ij}(t)]\\ 
+x_i(t+\varepsilon ) & [c_0c_1\chi _0\chi_{-1}P_{ij}(t)+c_0c_{-1}\chi _{-1}P_{ij}(t+\varepsilon )+c_1\chi _{-1}(J_1)_{ij}(t)+c_{-1}\chi_{-1}(J_1)_{ji}(t+\varepsilon )]\\ 
+x_i(t+2\varepsilon ) & [c_{-1}c_1\chi _{-1}\chi
_{-2}P_{ij}(t+\varepsilon )]\\ 
\multicolumn{2}{l}{+(c_1\chi _0\chi _1(J_2)_j(t-\varepsilon )+c_0\chi _0(J_2)_j(t)+c_{-1}\chi
_0\chi _{-1}(J_2)_j(t+\varepsilon )+\chi _0(J_3)_j(t))=0}
\end{array}}
\label{mchomat3}
\end{equation}
for each $j\in \{1,\ldots ,d\}$, $\forall t\in [a,b]$, with summation on $i$ when repeated. This equation has been heavily used for numerical
experiments.

\subsection{Solving D.E.L. in the safety interval}
\label{Solving D.E.L. in the safety interval}

Given $t_0\in [a,b]$, we define the \it safety interval\rm~ as the segment $\mathcal{I}_S\subset\mathbb{N}$ such that 
\begin{equation}
n\in\mathcal{I}_S\mbox{ iff }t_0+(n-j)\varepsilon\in[a,b]\mbox{ for all }j\in\{0,\ldots,4N-1\}.
\label{safety}
\end{equation}
We convert now (\ref{mchomat}) into a linear recurrence in $\mathbb{C}^{4dN}$. For $n\in\mathcal{I}_S$, we set 
\begin{equation*}
\mathbf{v}_n=\begin{pmatrix}\mathbf{x}(t_0+n\varepsilon)\\
\mathbf{x}(t_0+(n-1)\varepsilon)\\
\vdots\\
\mathbf{x}(t_0+(n-4N+1)\varepsilon)
\end{pmatrix}
\in \mathbb{C}^{4dN}.
\end{equation*}
When $n\in\mathcal{I}_S$, every characteristic function occuring in (\ref{mchomat}) equals to 1. Then, there exists well-defined matrices $A_n\in\mathbb{C}^{4dN\times 4dN}$ and vectors $\mathbf{b}_n\in\mathbb{C}^{4dN}$, depending only on $n,\Box_\varepsilon,P,Q,J_1,J_2$ and $J_3$, such that (\ref{mchomat}) is equivalent to
\begin{equation}
\mathbf{v}_{n+1}=A_n\mathbf{v}_n+\mathbf{b}_n.
\label{systlin}
\end{equation}
The matrix $A_n$ is defined at this stage if $n,n+1\in\mathcal{I}_s$, and admits a block structure with $4N\times 4N$ blocks of size $d\times d$. On block rows $2,3,\ldots ,N$, the blocks are either identity blocks or zero blocks, and on block row 1, the blocks $B_{i,n}$, $i\in \{1,\ldots,4N\}$ will express the matricial coefficients in the equation derived from (\ref{mchomat}) by solving it w.r.t. $\mathbf{x}(t_0+n\varepsilon)$. In this way, the matrix $A_n$ is the block companion matrix of the matrix polynomial
\begin{center}
$I_dX^{4N}-B_{1,n}X^{4N-1}-B_{2,n}X^{4N-2}-\ldots-B_{4N-1,n}X-B_{4N,n}$.
\end{center}

For sake of clarity, if $\mathcal{L}$ is stationary, it turns out that those $4N$ blocks have the shape 
\begin{equation}
B_{i,n}=c^{\prime}_iI_d+c^{\prime\prime}_iP^{-1}J_1+c^{\prime\prime\prime}_iP^{-1}Q
\label{blocksBi}
\end{equation} 
where the constants $c^\prime_i,c^{\prime\prime}_i,c^{\prime\prime\prime}_i$ depend only on $i$ and the coefficients $c_j$.
Moreover, if $N=1$, the following formulas for $\tra{A_n}$ and $\tra{\mathbf{b}_n}$ display the general structures of $A_n$ and $\mathbf{b}_n$
\begin{equation}
\tra{A_n}=\begin{pmatrix}
-\frac{(c_1+c_{-1})c_0}{c_1c_{-1}}I_d-\frac{c_1-c_{-1}}{c_1c_{-1}}P^{-1}J_1 &
I_d & 0 & 0\\
-\frac{(c_0^2+c_1^2+c_{-1}^2)}{c_1c_{-1}}I_d-%
\frac{1}{c_1c_{-1}}P^{-1}Q & 0 & I_d & 0 \\
-\frac{(c_1+c_{-1})c_0}{c_1c_{-1}}I_d-\frac{c_{-1}-c_{1}}{c_1c_{-1}}P^{-1}J_1
& 0 & 0 & I_d\\ -I_d & 0 & 0 & 0 \end{pmatrix}  
\label{traA}
\end{equation}
\vspace{-0.5cm}
\begin{equation}
\tra{\mathbf{b}}_n=\begin{pmatrix}
-\frac{P^{-1}}{c_1c_{-1}}(J_2\Box_{-\varepsilon}(1)+J_3) & 0 & 0 &  0
\end{pmatrix}.
\end{equation}

\subsection{Conditions for D.E.L. to be well-posed}

Let us consider the problem of solving D.E.L. under Dirichlet conditions. In the following result, we deal with existence, uniqueness and determination  of the restrictions of the solutions of D.E.L. to the various grids $\mathcal{G}_{t_0,\varepsilon}$.
\begin{theorem}
Let $t_0\in[a,b]$ and $\varepsilon>0$.
\begin{itemize}
\item If $\{a,b\}\subset \mathcal{G}_{t_0,\varepsilon}$, either there does not exist any solution $\mathbf{x}_\varepsilon(t)$ on $[a,b]$, or the restriction of each solution to $\mathcal{G}_{t_0,\varepsilon}$ is uniquely determined by the vectors $\mathbf{d}_a$ and $\mathbf{d}_b$ in $\mathbb{C}^d$.
\item If for instance $\{a,b\}\cap\mathcal{G}_{t_0,\varepsilon}=\{a\}$, then the set of solutions $\mathbf{x}_\varepsilon:\mathcal{G}_{t_0,\varepsilon}\mapsto\mathbb{C}^d$ of (\ref{mchomat}) is in one-to-one correspondance with $\mathbb{C}^d$.
\item If $\{a,b\}\cap\mathcal{G}_{t_0,\varepsilon}=\emptyset$, then the set of solutions of (\ref{mchomat}) on $\mathcal{G}_{t_0,\varepsilon}$ is in one-to-one correspondance with $\mathbb{C}^d\times\mathbb{C}^d$.
\end{itemize}
\end{theorem}

\begin{proof}
We assume that $N=1$ only to be more explicit, the case $N>1$ having the same qualitative features. 
Let us suppose that $\{a,b\}\subset \mathcal{G}_{t_0,\varepsilon}$ and w.l.o.g. that $t_0=a$ and $b-a=M\varepsilon$ where $M\in\mathbb{N}^\star$. For the need of the proof, we pursue the construction of $A_n$ when $n\notin\mathcal{I}_S$. In that case, some characteristic functions occuring in (\ref{mchomat}) vanish, this relationship is no more of order $d$, and the sizes of $\mathbf{v}_n$ and $A_n$ must change.
We have $\mathbf{x}_\varepsilon (a)=\mathbf{d}_a$ and we set $\mathbf{x}_\varepsilon (a+\varepsilon )=\mathbf{d}_s\in \mathbb{C}^d$ which is introduced without being determined at this stage, firmly from recurrences. Plugging $t=a$ in recurrence (\ref{mchomat3}) and solving, we first get \begin{center}
$\mathbf{x}_\varepsilon (a+2\varepsilon)=B_{1,1}\mathbf{d}_s+B_{2,1}{\mathbf{d}_a}$
\end{center}
where $B_{1,1},B_{2,1}$ are blocks similar to those occuring in (\ref{blocksBi}). Next, with $t=a+\varepsilon $ we find 
\begin{center}
$\mathbf{x}_\varepsilon (a+3\varepsilon )=(B_{1,2}B_{1,1}+B_{2,2})\mathbf{d}_s+(B_{1,2}B_{2,1}+B_{3,2})\mathbf{d}_a$.
\end{center}
The following iterations express $\mathbf{x}_\varepsilon (a+n\varepsilon )$ as a linear combination of the vectors $\mathbf{d}_s,\mathbf{d}_a$, with coefficients being polynomial matrices in $B_{i,k}$. From index from $n=4$ to $n=M-3$ the recurrence (\ref{mchomat3}) becomes or order $4N+1$ and may be reformulated as (\ref{systlin}). Finally, the three last steps  $n=M-2,M-1,M$ are similar and imply three systems of decreasing sizes. In order to convert matricially this process, we introduce the five rectangular matrices $A_i$
\[
A_1=\left( 
\begin{array}{cc}
B_{1,1} & B_{2,1} \\ 
I_d & 0 \\ 
0 & I_d
\end{array}
\right) ,~~A_2=\left( 
\begin{array}{ccc}
B_{1,2} & B_{2,2} & B_{3,2} \\ 
I_d & 0 & 0 \\ 
0 & I_d & 0 \\ 
0 & 0 & I_d
\end{array}
\right),~A_{M-1}=\left( B_1~-I_d\right),
\]
\[
A_{M-3}=\left( 
\begin{array}{cccc}
B_{1,M-3} & B_{2,M-3} & B_{3,M-3} & -I_d \\ 
I_d & 0 & 0 & 0 \\ 
0 & I_d & 0 & 0
\end{array}
\right),~A_{M-2}=\begin{pmatrix}
B_{1,M-2} & B_{2,M-2} & -I_d \\ 
I_d & 0 & 0
\end{pmatrix}
\]
The operators $A_1,A_2$ are used to compute the values $\mathbf{x}_\varepsilon (a+n\varepsilon )$ for $n=2,3$ linearly as functions of $\mathbf{d}_a$, $\mathbf{d}_s$. Next, we have 
\begin{equation}
\mathbf{x}_\varepsilon (a+n\varepsilon )=(A_{n-1}\ldots A_3)A_2A_1\left( 
\begin{array}{c}
\mathbf{d}_s \\ 
\mathbf{d}_a 
\end{array}
\right)  \label{solint}
\end{equation}
for $4\leq n\leq M-3$. Finally, $A_{M-3},A_{M-2},A_{M-1}$ are used to find $\mathbf{x}_\varepsilon (a+n\varepsilon )$ for $M-2\leq n\leq M$. At the end of the process, we get the shooting equation for the vector $\mathbf{d}_s$:
\begin{equation}
\mathbf{d}_b=\mathbf{x}_\varepsilon(b)=\mathbf{x}_\varepsilon (a+M\varepsilon
)=A_{M-1}A_{M-2}A_{M-3}(A_{M-4}\ldots A_3)A_2A_1\left( 
\begin{array}{c}
\mathbf{d}_s \\ 
\mathbf{d}_a
\end{array}
\right).  
\label{solfin}
\end{equation}
Now, existence and unicity of the restriction of $\mathbf{x}_\varepsilon$ to the grid $\mathcal{G}_{a,\varepsilon}=\mathcal{G}_{b,\varepsilon}$ is equivalent to the fact that the shooting method is successful, that is 
\begin{equation}
\det\left(A_{M-1}A_{M-2}\ldots A_2A_1\left( 
\begin{array}{c}
I_d\\ 
0_d
\end{array}
\right)\right)\neq 0.
\label{cramer}
\end{equation}

Let us consider now the cases where $(b-a)/\varepsilon$ is not an integer so that $|\mathcal{G}_{t_0,\varepsilon}\cap\{a,b\}|<2$, the previous matrix formalism being similar. 
If $\{a,b\}\cap\mathcal{G}_{t_0,\varepsilon}=\{a\}$, then any vector $\mathbf{d}_s\in\mathbb{C}^d$ determines a solution $\mathbf{x}_\varepsilon(t)$ on $\mathcal{G}_{t_0,\varepsilon}$. The case $\{a,b\}\cap\mathcal{G}_{t_0,\varepsilon}=\{b\}$ is entirely similar and we have infinitely many choices for $\mathbf{d}_s=\mathbf{x}_\varepsilon(b-\varepsilon)$. 
Lastly, if $\{a,b\}\cap\mathcal{G}_{t_0,\varepsilon}=\emptyset$ then we first may choose arbitrarily the two vectors $\mathbf{x}_\varepsilon(\min\mathcal{G}_{t_0,\varepsilon})$ and $\mathbf{x}_\varepsilon(\varepsilon+\min\mathcal{G}_{t_0,\varepsilon})$ in $\mathbb{C}^d$ and we use (\ref{mchomat}) to compute iteratively the values of $\mathbf{x}_\varepsilon$ on $\mathcal{G}_{t_0,\varepsilon}$.
\end{proof}

\begin{remark}\rm
Let us note that if $\mathcal{L},a,b,$ are fixed, the underlying determinant of $A_{M-4}\ldots A_3$ is a nonzero polynomial of degree less than $2d\times(4M-6)$ w.r.t. the coefficients $c_{-1},c_0,c_1$ and does not vanish generically. 
\end{remark}

\subsection{Eigenvectors of the matrix $A_n$ when $n\in \mathcal{I}_S$}

As it is the case for the sequences of vectors satisfying ordinary linear recurrences, the qualitative features of the solution 
$\mathbf{x}_\varepsilon(t_0+n\varepsilon)$ of D.E.L. are reflected by properties of the spectrum $Sp(A_n)$ of $A_n$.
\begin{proposition}
The eigenvectors of $A_n$ in $\mathbb{C}^{4dN}$ have the shape 
\begin{center}
$\tra{\mathbf{v}}=\tra(\mathbf{w}\lambda ^{4N-1},\mathbf{w}\lambda^{4N-2},\ldots ,\mathbf{w})$,
\end{center}
where $\displaystyle \mathbf{w}\in \ker(\sum_{i=1}^{4N}B_{i,n}\lambda ^{4N-i}-\lambda ^{4N}I_d)\subset \mathbb{C}^d$. We have $\det (A_n)=(-1)^d$ and
\begin{center}
$\displaystyle \det (A_n-\lambda I_{4dN})=\det(\sum_{i=1}^{4N}B_{i,n}\lambda ^{4N-i}-\lambda ^{4N}I_d)$.
\end{center}
\label{propvp}
\end{proposition}

\begin{proof}
The two results are well known in the scalar case $d=1$. Let us give some details when we deal with characteristic functions and $d>1$.\\
If $\mathbf{v}\in\mathbb{C}^{4dN}$ is an eigenvector of $A_n$ associated to $\lambda\in\mathbb{C}$, we partition it as $\tra{\mathbf{v}}=\tra{(\mathbf{w}_{4N},\ldots,\mathbf{w}_1)}$ where $\mathbf{w}_i\in\mathbb{C}^d$. We next identify the corresponding blocks of size $d\times 1$ in $A_n\mathbf{v}=\lambda\mathbf{v}$ to get $\mathbf{w}_{i}=\lambda\mathbf{w}_{i-1}=\lambda^{i-1}\mathbf{w}_1$ for $2\leq i\leq 4N$. Renaming $\mathbf{w}_1$ as $\mathbf{w}$ and plugging the vectors $\mathbf{w}_i$ in the first block row of $A_n\mathbf{v}$ yield the first property.\\ The second property may be easily proved by using matricial techniques for partitioned matrices (see for instance \cite[pp.~36]{Zhang}).
\end{proof}

\section{Pseudo-periodic solutions of C.E.L. and D.E.L. for stationary lagrangians}

In this section, the datas $N,d,\mathbf{d}_a,\mathbf{d}_b,\varepsilon$ are fixed but arbitrary, and $\mathcal{L}$ is stationary. 
We will say that $\mathcal{L}$ is a stationary non-resonant oscillatory lagrangian w.r.t. the datas $N,d,\mathbf{d}_a,\mathbf{d}_b,\varepsilon,\Box_\varepsilon$ if and only if D.E.L. and C.E.L. admit one and only one pseudo-periodic solution $\mathbf{y}_\varepsilon(t)$ and $\mathbf{z}(t)$ respectively.

\subsection{Solving C.E.L.}

Let us study first the existence, unicity and periodicity or pseudo-periodicity of the solutions of (\ref{elho}). 
\begin{proposition}
Suppose that $\mathcal{L}$ is stationary and that for some matrices $\Omega_1,\Omega_2\in\mathbb{C}^{d\times d}$ we have 
\begin{equation}
P\Omega^2+2iJ_1\Omega+Q=0\mbox{ and }\det(\exp{(i(b-a)\Omega_2)}-\exp{(i(b-a)\Omega_1)})\neq 0.
\label{NROC}
\end{equation}
Then, for all $\mathbf{d}_a,\mathbf{d}_b\in\mathbb{C}^{d}$ there exists one and only one solution of C.E.L. (\ref{elho}) together with Dirichlet boundary conditions. Moreover, if $\Omega_1$ and $\Omega_2$ are diagonalizable, each component $f(t)$ of $\mathbf{z}(t)$ may be written as
\begin{equation}
f(t)=cst+\sum_{\footnotesize \begin{array}{c}k=a,b\\j\in[1,d]\end{array}}\sum_{\omega\in Sp(\Omega_1)\cup Sp(\Omega_2)}cst_{k,j,\omega}(\mathbf{d}_k)_j\exp({i\omega(t-a)}).
\label{solcont}
\end{equation}
where the various constants depend only on their indices as well as $b-a$ and the eigenvalues of $\Omega_1$ and $\Omega_2$.
\label{prop1}
\end{proposition}

\begin{proof}
We see first that  
\begin{equation}
\mathbf{z}(t)=\exp{(i(t-a)\Omega_1)}\mathbf{z}_1+\exp{(i(t-a)\Omega_2)}\mathbf{z}_2-Q^{-1}J_3,  
\label{hos}
\end{equation}
is a solution of (\ref{elho}) for all $\mathbf{z}_1,\mathbf{z}_2\in\mathbb{C}^d$. In order to fit the Dirichlet conditions, the vectors $\mathbf{z}_1,\mathbf{z}_2$ must satisfy $\mathbf{z}_1+\mathbf{z}_2=\mathbf{d}_a+Q^{-1}J_3$ and $\exp{(i(b-a)\Omega_1)}\mathbf{z}_1+\exp{(i (b-a) \Omega_2)}\mathbf{z}_2=\mathbf{d}_b+Q^{-1}J_3$. Due to (\ref{NROC}), the previous system is Cramer and the solution is equal to $\mathbf{z}_1=R\mathbf{e}_2$ and $\mathbf{z}_2=-R\mathbf{e}_1$ where $R\in\mathbb{C}^{d\times d}$ and $\mathbf{e}_1,\mathbf{e}_2\in\mathbb{C}^d$ are respectively defined by
\begin{center}
$R=(\exp({i (b-a)\Omega_2})-\exp({i(b-a)\Omega_1}))^{-1}$ and $\mathbf{e}_k=\exp({i(b-a)\Omega_k})\mathbf{d}_a-\mathbf{d}_b+(\exp({i (b-a) \Omega_k})-I_d)Q^{-1}J_3$.
\end{center}
By considering the previous formulas, we see that each component $f(t)$ of $\mathbf{z}(t)$ depends linearly on $(\mathbf{d}_a,\mathbf{d}_b,Q^{-1}J_3)\in\mathbb{C}^{3d}$ and may be returned as (\ref{hos}) where the constants do not depend on  $t,\mathbf{d}_a,\mathbf{d}_b$ nor on $Q^{-1}J_3$. Indeed, since $\Omega_k$ is diagonalizable for $k=1,2$, each entry in $\exp(it\Omega_k)$ is a monomial exponential w.r.t. $t$. Thus, each component of (\ref{hos}) has the shape (\ref{solcont}).
\end{proof}

As (\ref{hos}) shows, the solution $\mathbf{z}(t)$ of C.E.L. is pseudo-periodic if and only if the entries of $\Omega_1$ and $\Omega_2$ are real. If  $\mathcal{L}$ is real-valued, that is to say all the coefficients in (\ref{bilineaire}) are real, pseudo-periodicity is equivalent to $J_1=0$ and $-P^{-1}Q=\Omega^2$ for some $\Omega\in\mathbb{R}^{d\times d}$. In that case, the function $\mathbf{z}(t)$ may be returned as 
\begin{center}
$\mathbf{z}(t)=\cos{(t\Omega)}\mathbf{z}_1+\sin{(t\Omega)}\mathbf{z}_2-Q^{-1}J_3$,
\end{center}
so that the second assumption in (\ref{NROC}) reads as
\begin{equation}
\det\begin{pmatrix}\cos a\Omega & \sin a\Omega\\\cos b\Omega & \sin
b\Omega\end{pmatrix}\neq 0.  \label{resocel}
\end{equation} 
\begin{remark}\rm
The extension to the case $P^{-1}Q=+\Omega^2$ and $J_1=0$ is straightforward and in this case the formula involves $\cosh (t\Omega )$ and $\sinh (t\Omega )$ in $\mathbf{z}(t)$.
\end{remark}
\begin{remark}\rm
Periodicity of $\mathbf{z}(t)$ is obviously equivalent to $\exp(iT\Omega_k)=I_d$, for some $T>0$ and for all $k=1,2$.
\end{remark}
\begin{remark}\rm
The problem of existence of square or higher roots to real or complex matrices, as in (\ref{NROC}), has led to huge bibliography. For instance, a simple criterion depending on elementary divisors for a real nonsingular matrix $M$ to have real square roots is that each elementary divisor corresponding to a negative eigenvalue occurs an even number of times, see \cite[pp.~413, Theorem 5]{Hig}. But this result has been improved by Higham, since he proved that at most $2^{r+c}$ real square roots of a real nonsingular matrix $M$ may be expressed as some polynomial in $M$ (\cite[pp. 416, Theorem 7]{Hig}), $r$ (resp. $c$) being the number of real (resp. distinct complex conjugate pair of) eigenvalues of $M$.
\end{remark}

\subsection{Generation of pseudo-periodic solutions of D.E.L.}

Let us study the existence, unicity and pseudo-periodicity of the solutions of (\ref{mchomat}). In order to express the components of the solution of D.E.L. as in (\ref{solcont}), we use the main results in Section 3 by adding the assumption that $\mathcal{L}$ is stationary. In that case, for all $n\in\mathcal{I}_S$ defined in (\ref{safety}), the matrix $A_n$ and the vector $\mathbf{b}_n$ do not depend on $n$. We set $A=A_n$ and $\mathbf{b}=\mathbf{b}_n$ for $n\in\mathcal{I}_S$.
\begin{proposition}
We suppose that  
\begin{center}
$|Sp(A)|=4Nd$,\hspace{0.5cm}$1\notin Sp(A)$,\hspace{0.5cm}$\displaystyle M=\frac{b-a}{\varepsilon}\in\mathbb{N}$,
\end{center}
and (\ref{cramer}) holds. Then the restriction of any solution $\mathbf{x}_\varepsilon$ to $\mathcal{G}_{a,\varepsilon}\cap[a+2N\varepsilon,b-2N\varepsilon]$ is uniquely determined and its components have the shape 
\begin{equation}
g_\varepsilon (t)=cst+\sum_{\footnotesize \begin{array}{c}k=a,b\\j\in[1,d]\end{array}}\sum_{\exp(i\theta)\in Sp(A)}cst_{k,j,\theta}(\mathbf{d}_k)_j\exp\left(\frac{i\theta}{\varepsilon}(t-a)\right).
\label{soldisc}
\end{equation}
Moreover, the restriction of $\mathbf{x}_\varepsilon$ on $\mathcal{G}_{a,\varepsilon}\cap[a+2N\varepsilon,b-2N\varepsilon]$ is pseudo-periodic if and only if $Sp(A)\subset \mathbb{U}$. 
\label{prop2}
\end{proposition}
\begin{proof}
Let us define the two vectors $J_5$ and $J_6$ in $\mathbb{C}^d$ by :
\begin{equation}
J_5=-\frac 1{c_Nc_{-N}}P^{-1}(J_2\Box _{-\varepsilon} (1)+J_3),\hspace{0.5cm}J_6=(I_d-\sum_{i=1}^{4N}B_{i,n})^{-1}J_5.
\end{equation}
Note that $J_6$ is well-defined since $1\notin Sp(A)$. When $n\in\mathcal{I}_S$, formula (\ref{mchomat}) may be rewritten under the form 
\begin{equation}
\displaystyle \mathbf{x}(t_0+n\varepsilon)=\sum_{i=1}^{4N}B_{i,n}\mathbf{x}(t_0+(n-i)\varepsilon)+J_5. 
\label{recvector}
\end{equation}
A particular constant solution of (\ref{recvector}) is obviously given by $\mathbf{x}(t_0+n\varepsilon)=J_6$. We now apply Proposition \ref{propvp}. Given $\mathbf{w}\in\mathbb{C}^d$ and $\mathbf{v}\in\mathbb{C}^{4Nd}$, the vector sequence $(\lambda^l\mathbf{w})_l$ satisfies the homogeneous recurrence (\ref{recvector}) if and only if $\lambda\in Sp(A)$ is associated to $\mathbf{v}$. Since $A$ is diagonalizable, the eigenvectors are linearly independent and we get 
\begin{equation}
\mathbf{x}_\varepsilon(t_0+n\varepsilon )=\sum_j\lambda _j^n\mathbf{w}_j+J_6,  
\label{xdsansbord}
\end{equation}
where $\mathbf{w}_j$ are appropriate vectors in $\mathbb{C}^d$. If $\varepsilon=(b-a)/M$, $\mathbf{d}_s$ is a well-defined linear combination of $\mathbf{d}_a$ and $\mathbf{d}_b$, as seen in (\ref{solfin}). Let us introduce the linear system of $4Nd$ equations
\begin{center}
$\displaystyle \sum_{j=1}^{4Nd}\lambda_j^n\mathbf{w}_j=\mathbf{x}_\varepsilon(a+n\varepsilon)-J_6$
\end{center}
where the r.h.s. are computed from (\ref{solint}). The determinant of this system is the Vandermonde $V(\lambda_1,\ldots,\lambda_{4Nd})$ which is nonzero since the eigenvalues $\{\lambda _i\}_i$ of $A$ are pairwise distinct. Hence, due to (\ref{solfin}), the vectors $\mathbf{w}_1,\ldots,\mathbf{w}_{4Nd}$ are well-defined and may be uniquely written as linear combinations of $\mathbf{d}_a$ and $\mathbf{d}_b$. If we denote the eigenvalues of $A$ by $\lambda=\exp(i\theta)$ with $\theta\in\mathbb{C}$, (\ref{xdsansbord}) may be rewritten as (\ref{soldisc}). Pseudo-periodicity is equivalent to the requirement that $\theta\in\mathbb{R}$ for all $\exp(i\theta)\in Sp(A)$, that is $Sp(A)\subset\mathbb{U}$.
\end{proof}

\begin{proposition}
Under the assumptions of Proposition \ref{prop2} and the hypothesis $Sp(A)\subset\mathbb{U}$ and $\varepsilon<(b-a)/(4N(d+1))$, we may associate to any solution $\mathbf{x}_\varepsilon:[a,b]\rightarrow\mathbb{C}^d$ of D.E.L. one and only one function $\mathbf{y}_\varepsilon:[a,b]\rightarrow\mathbb{C}^d$ such that 
\begin{itemize}
\item $\mathbf{y}_\varepsilon$ is a solution of D.E.L. on $[a,b]$,
\item $\mathbf{y}_\varepsilon$ is pseudo-periodic on $[a,b]$,
\item $\mathbf{x}_\varepsilon$ and $\mathbf{y}_\varepsilon$ agree on $\mathcal{G}_{a,\varepsilon}\cap[a+2N\varepsilon,b-2N\varepsilon]$.
\end{itemize}
If $\mathbf{x}_\varepsilon$ is pseudo-periodic then $\mathbf{y}_\varepsilon=\mathbf{x}_\varepsilon$. Moreover, if $\mathbf{x}_\varepsilon$ is continuous on $[a,b]$, then, for all $\delta>0$, $\sup_{t\in [a+\delta,b-\delta ]}\Vert \mathbf{x}_\varepsilon (t)-\mathbf{y}_\varepsilon (t)\Vert$ tends to 0 as $\varepsilon $ tends to 0.
\label{prop3}
\end{proposition}
\begin{proof}
Indeed, $\mathbf{y}_\varepsilon$ is generated by using (\ref{soldisc}) outside the grid and outside $[a+2N\varepsilon,b-2N\varepsilon]$, so that obviously $\mathbf{x}_\varepsilon=\mathbf{y}_\varepsilon$ on $\mathcal{G}_{a,\varepsilon}\cap[a+2N\varepsilon,b-2N\varepsilon]$. It turns out that $\mathbf{y}_\varepsilon$ is also a solution of D.E.L. since the coefficients of the recurrence in (\ref{mchomat}) are independent on time, that is to say the coefficients are the same for any grid. Due to the assumption $Sp(A)\subset\mathbb{U}$, $\mathbf{y}_\varepsilon$ is pseudo-periodic. Let us prove the unicity~: we assume that there exists two pseudo-periodic solutions $\mathbf{y}_{\varepsilon,1}$ and $\mathbf{y}_{\varepsilon,2}$. Let us fix $k\in [1,d]$. The component of index $k$ of $\mathbf{y}_{\varepsilon,2}-\mathbf{y}_{\varepsilon,1}$ is of the shape (\ref{soldisc}). So we may define $\delta_p\in\mathbb{C}$ as the coefficient of $\exp(i\theta_p(t-a)/\varepsilon)$ in $\mathbf{y}_{\varepsilon,2}(t)-\mathbf{y}_{\varepsilon,1}(t)$ for all $p\in[1,4Nd]$. Suppose now that $\varepsilon<(b-a)/(4N(d+1))$. Setting $t=a+(2N+n)\varepsilon$ in (\ref{soldisc}) with $1\leq n\leq 4Nd$ we get a linear system of size $4Nd$ such as
\begin{center}
$\displaystyle \sum_{k=1}^{4Nd}\mathbf{\delta}_k\exp(in\theta_k)=(\mathbf{y}_{\varepsilon,2}(a+(2N+n)\varepsilon)-\mathbf{y}_{\varepsilon,1}(a+(2N+n)\varepsilon))_j=0$.
\end{center}
By assumption, the Vandermonde determinant of this system is nonzero and we get $\mathbf{\delta}_p=0$ for all $p$. Since this holds for all component of $\mathbf{y}_{\varepsilon,2}(t)-\mathbf{y}_{\varepsilon,1}(t)$, we get unicity that is $\mathbf{y}_{\varepsilon,1}(t)=\mathbf{y}_{\varepsilon,2}(t)$ for all $t\in[a,b]$. As a consequence of unicity, if $\mathbf{x}_\varepsilon$ is itself pseudo-periodic, then $\mathbf{y}_\varepsilon=\mathbf{x}_\varepsilon$.

Finally, let us choose $\varepsilon$ so that $2N\varepsilon <\delta $. Since $\mathbf{x}_\varepsilon$ and $\mathbf{y}_\varepsilon$ are uniformly continuous on $[a+\delta ,b-\delta ]$, we choose $\varepsilon$ less than a modulus of uniform continuity for $\delta/2$. If $t\in [a+\delta ,b-\delta]$ and $t_G=a+n\varepsilon $ is the closest point of the grid to $t$, the triangle inequality yields 
$\Vert \mathbf{x}_\varepsilon (t)-\mathbf{y}_\varepsilon (t)\Vert \leq \Vert \mathbf{x}_\varepsilon (t)-\mathbf{x}_\varepsilon (t_G)\Vert +\Vert \mathbf{y}_\varepsilon (t_G)-\mathbf{y}_\varepsilon (t)\Vert\leq \delta$. \end{proof}

\begin{remark}\rm
If the coefficients $c_i$ are chosen as $\gamma_i/\varepsilon$ then the matrix $A$ is a quadratic polynomial w.r.t. $\varepsilon$. The eigenvalues of $A$ are algebraic functions of $\varepsilon$. Determining if $Sp(A)$ is included in $\mathbb{U}$ is a polynomial elimination problem. For instance, if $N=d=1$, the operators $\Box_\varepsilon$ for which the spectrum of $A$ is included in the unit circle are of the shape $\Box_\varepsilon=\Box_\varepsilon^{[\frac{1}{2},\frac{1}{2}]}+ik\Box_\varepsilon^{[1,-1]}$ where $k\in\mathbb{R}$, see \cite[pp.7, Proposition 5.2]{RS2}.
\end{remark}

\section{Obstructions to convergence of $\mathbf{y}_\varepsilon(t)$ to $\mathbf{z}(t)$ as $\varepsilon$ tends to 0}

\subsection{Preliminary discussion}

Under the assumptions of the three propositions of the previous section, to prove that
\begin{center}
\it $\mathbf{y}_\varepsilon(t)$ tends to $\mathbf{z}(t)$ uniformly locally on $]a,b[$ as $\varepsilon$ tends to 0\rm,
\end{center}
is not an easy task. It relies on the comparison of the formulas (\ref{solcont}) and (\ref{soldisc}). This is why we focus on phases and amplitudes occuring in $\mathbf{z}(t)$ and $\mathbf{y}_\varepsilon(t)$. 

The convergence of $\mathbf{y}_\varepsilon(t)$ to $\mathbf{z}(t)$ as $\varepsilon$ tends to 0 is related to the three following properties.
\begin{enumerate}
\item[(a)] If $\lambda_j$ is an eigenvalue of $A$ which tends to 1 as $\varepsilon$ tends
to 0, its phase $\theta_j$ is such that $\frac{\theta_j}{\varepsilon}$ tends to a phase $\omega_k$ of some eigenvalue $\exp(i\omega_k)$ of $\Omega$. 
\item[(b)] For any phase $\omega\in\mathbb{R}$ such that $\exp(i\omega)\in Sp(\Omega)$, let $f_\omega\in\mathbb{C}$ be the amplitude of $\exp(i\omega(t-a))$ in (\ref{solcont}). Similarly, let $g_{\varepsilon,\omega}\in\mathbb{C}$ be the sum of the amplitudes occuring in (\ref{soldisc}) corresponding to $\exp(i\theta)\in Sp(A)$ with $\theta/\varepsilon \sim\omega$ as $\varepsilon\rightarrow 0$. Then $\lim\limits_{\varepsilon\to 0}|g_{\varepsilon,\omega}-f_\omega|=0$.
\item[(c)] The sum of the contribution $g_\theta$ in (\ref{soldisc}) of eigenvalues $\exp(i\theta)\in Sp(A)$ not tending to 1 cancels, as $\varepsilon$ tends to 0.
\end{enumerate}
Summing all triangle inequalities 
\[
\left|g_{\varepsilon,\omega}\exp{\left(i\theta \frac{t-a}{\varepsilon}\right)}-f_\omega \exp{(i\omega (t-a))}\right|\leq |g_{\varepsilon,\omega}-f_\omega|+|f_\omega|\left|\frac
\theta \varepsilon -\omega \right|(b-a)
\]
over the group of eigenvalues tending to 1, and considering the contribution of eigenvalues which are not tending to 1, $\|\mathbf{y}_\varepsilon(t)-\mathbf{z}(t)\|_{\tiny \mathcal{L}^\infty([a+\delta,b-\delta])}$ is upper bounded by
\begin{equation*}
\sum_{\tiny \begin{array}{c}e^{i\theta}\in Sp(A)\\\theta\not\rightarrow 0\end{array}}|g_\theta|+\sum_{\tiny \begin{array}{c}e^{i\theta}\in Sp(A)\\e^{i\omega}\in Sp(\Omega)\\\theta\simeq\varepsilon\omega\rightarrow 0\\\end{array}}|g_{\varepsilon,\omega}-f_\omega|+M\varepsilon\sum_{e^{i\omega}\in Sp(\Omega)}|f_\omega|\sup_{\tiny \begin{array}{c}e^{i\theta}\in Sp(A)\\e^{i\omega}\in Sp(\Omega)\\\theta\simeq\varepsilon\omega\rightarrow 0\\\end{array}}\left|\frac{\theta}{\varepsilon}-\omega\right|
\end{equation*}
If the three properties hold, the previous bound tends to 0 as $\varepsilon$ tends to 0. Note lastly, that the result of convergence itself is related to the success of the shooting method and the convergence of the scheme.

We shall illustrate in the following two subsections the convergence issue by giving two convenient examples when $N=1$ and $d=2$ for two special cases of $\Box_\varepsilon^{[r,s]}$. In that case, we denote $\mathbf{x}_{n}=\mathbf{x}(a+n\varepsilon)=(x_{n},y_{n})$.  

\label{subsection5.1}

\subsection{First example}

Let us consider $r\in\mathbb{R}$ and $\Box_\varepsilon^{[r,r]}$ defined as in (\ref{ourboxpq}). We restrict ourselves to the case where $P=\begin{pmatrix}p_1&p_2\\p_2&p_1\end{pmatrix}$, $Q=\begin{pmatrix}q_1&q_2\\q_2&q_1\end{pmatrix}$ and $J_1=0$. The condition (\ref{NROC}) implies that the two numbers $\displaystyle \frac{q_1+q_2}{p_1+p_2}$ and $\displaystyle \frac{q_1-q_2}{p_1-p_2}$ are negative. In that case, we find that $\displaystyle Sp(\Omega)=\left\{\sqrt{\left|\frac{q_1+q_2}{p_1+p_2}\right|},\sqrt{\left|\frac{q_1-q_2}{p_1-p_2}\right|}\right\}$.\\
The recurrence $\mathbf{v}_{n+1}=A\mathbf{v}_n+\mathbf{b}$ splits into two recurrences for $\mathbf{x}_{2n}$ and $\mathbf{x}_{2n+1}$. We note that
\begin{equation} |\det(\exp(i\Omega(b-a))-\exp(-i\Omega (b-a) ))|=4\left|\sin\left(\frac{1}{2}\sqrt{\left|\frac{q_1+q_2}{p_1+p_2}\right|}\right)\sin\left(\frac{1}{2}\sqrt{\left|\frac{q_1-q_2}{p_1-p_2}\right|}\right)\right|
\label{shoot}
\end{equation}
and accordingly to Proposition \ref{prop1}, the shooting method is successful for $\mathbf{z}(t)$ if and only if the two eigenvalues of $\Omega$ are not commensurable with $\pi$.
We get so far 
\vspace{0.1cm}
\begin{center}
$\left\{\begin{array}{rcl}
x_{n+2} & = & \displaystyle \left(2+\frac{\varepsilon^2}{r^2}\frac{p_1q_1-p_2q_2}{p_1^2-p_2^2}\right)x_n+\frac{\varepsilon^2}{r^2}\frac{p_1q_2-p_2q_1}{p_1^2-p_2^2}y_{n}-x_{n-2},\vspace{0.1cm}\\
y_{n+2} & = & \displaystyle \frac{\varepsilon^2}{r^2}\frac{p_1q_2-p_2q_1}{p_1^2-p_2^2}x_n+\left(2+\frac{\varepsilon^2}{r^2}\frac{p_1q_1-p_2q_2}{p_1^2-p_2^2}\right)y_{n}-y_{n-2}.
\end{array}\right.$
\end{center}
\vspace{0.1cm}
The coefficients occuring in the previous recurrence are the entries of block $B_{2,n}$ defined as in (\ref{blocksBi}). Note by the way that the two  blocks $B_{1,n}$ and $B_{3,n}$ are zero.  
The sequences $((x_{2n},y_{2n}))$ and $((x_{2n+1},y_{2n+1}))$ obey to the same recurrence but are computed independently each to the other. If $M$ is even, the Dirichlet conditions for $n=0$ and $n=M$ ensure existence and unicity of $((x_{2n},y_{2n}))$ provided the shooting method is successful.     
By reordering the components of the vector $\mathbf{v}_n$, the matrix $A$ is equivalent to a block diagonal matrix $\begin{pmatrix}K_4&0_4\\0_4&K_4\end{pmatrix}$ where $K_4=\begin{pmatrix}B_{2,n} & -I_2\\I_2 & 0\end{pmatrix}$ and $0_k$ is the zero matrix of size $k$. Now the spectrum of $K_4$ consists of the four numbers $\exp(\pm i\theta _1)$, $\exp(\pm i\theta _2)$ where
\begin{center}
$\displaystyle\theta_1=\arccos \left(1-\frac{\varepsilon^2}{2r^2}\left|\frac{q_1+q_2}{p_1+p_2}\right|\right)$ and $\displaystyle\theta_2=\arccos
\left(1-\frac{\varepsilon^2}{2r^2}\left|\frac{q_1-q_2}{p_1-p_2}\right|\right)$.
\end{center}
We have here $|Sp(A)|=4$ and $4Nd=8$, so Proposition (\ref{prop2}) does not apply and indeed, the sequence $((x_{2n+1},y_{2n+1}))$ is not uniquely determined. Lastly, we get
\begin{center}
$\displaystyle\theta_{1,2}\sim \frac{\varepsilon }{r}\sqrt{\left| \frac{q_1\pm q_2}{p_1\pm p_2}\right| }\sim\frac{\varepsilon}{2r} \omega _{1,2}$.
\end{center}
We note that property (a) holds if and only if $\displaystyle r=\frac{1}{2}$. The property (b) is much more delicate and is discussed in the last section. At last, property (c) is obviously true.

\subsection{Second example}

We consider the operator used by Cresson in \cite{Cre} to define scale derivatives :
\begin{equation*} \Box_{\varepsilon}^{[\frac{1-i}{2},\frac{1+i}{2}]}=-\chi_{1}(t)\frac{1+i}{2\varepsilon}\mathbf{x}(t-\varepsilon)+\frac{i}{\varepsilon}\mathbf{x}(t)+\chi_{-1}(t)\frac{1-i}{2\varepsilon}\mathbf{x}(t+\varepsilon).
\end{equation*} 
The characteristic polynomial of $A$ may be factored into two biquadratic equations. The eight eigenvalues of $A$ may be written as 
\[
\lambda =\frac{1+\zeta _1\sqrt{1+2(\varepsilon \omega _k)^2}}2+i\zeta_2\frac{\sqrt{2}}2\sqrt{1-(\varepsilon \omega _k)^2-\zeta _1\sqrt{%
1+2(\varepsilon \omega _k)^2}},
\]
where $\omega _k$ ($k=1,2$) are the eigenvalues of the matrix $\Omega $ and $\zeta _1^2=\zeta _2^2=1$. We see that the eigenvalues of $A\in\mathbb{C}^{8\times 8}$ are all
distinct and of modulus 1. Looking for the limits as $\varepsilon $ tends to 0 of the eigenvalues, we get four limits equal to 1, two equal to $i$ and
two equal to $-i$. The first four eigenvalues check the property (a) as shows expansion with Taylor series w.r.t. $\varepsilon\omega_k$. Note that the four eigenvalues tending to 1 (obtained by choosing $\zeta_1=1$) may be written as $\lambda=\exp({i\zeta_2\omega_k})+o(1)$.\\
The sub-sum of the terms in (\ref{soldisc}) implying eigenvalues which tend to 1 as $\varepsilon $ tends to 0 may be rewritten as  
\[
\displaystyle \sum_{j=1,2,\zeta =\pm 1}cst_{\zeta ,j}(1+i\omega
_j\varepsilon \zeta -\frac 12\omega _j^2\varepsilon ^2+O(\varepsilon ^3))^{\frac{t-a}\varepsilon },\mathrm{\ }
\]
where $O(\varepsilon ^3)$ is uniform in $t$, by using Puiseux expansion of each factor of $\det (A-\lambda I_8)$ around $\varepsilon =0$.
The limit of this sum as $\varepsilon $ tends to 0, is a combination of $\exp(i\omega _k(t-a))$ and $\exp(-i\omega _k(t-a))$ which is a step towards property (b).

In fact, if we want to  justify the choice of $\Box_\varepsilon^{[\frac{1-i}{2},\frac{1+i}{2}]}$, we may generalize a little bit the previous calculations to $\Box_\varepsilon$ such that the operator of the l.h.s. of (\ref{mchomat}) converges to (\ref{elho}) and is such that $Sp(A)\subset\mathbb{U}$. In \cite[Proposition 5.2]{RS2}, we prove that $\Box_\varepsilon$ must be of the shape $\Box_\varepsilon^{[r,s]}$ with $r=\frac{1}{2}-it$ and $s=\frac{1}{2}+it$. A formal computation of the eigenvalues $\lambda_j\in Sp(A)$ shows that for two indices $j_1$ and two indices $j_2$ we have
\begin{center}
$\displaystyle \lim\limits_{\varepsilon\to 0}\lambda_{j_1}(\varepsilon)=1+sgn(1-2t)\frac{it}{\frac{1}{4}+t^2}$ and $\displaystyle \lim\limits_{\varepsilon\to 0}\lambda_{j_2}(\varepsilon)=1-sgn(1-2t)\frac{it}{\frac{1}{4}+t^2}$.
\end{center}
The two operators $\Box_\varepsilon^{[\frac{1}{2},\frac{1}{2}]}$ and $\Box_\varepsilon^{[\frac{1-i}{2},\frac{1+i}{2}]}$ are the only ones such that 
\begin{center}
$N=1$, $\Box_\varepsilon(1)=0$, $\Box_\varepsilon(t)=1$ in $[a+2N\varepsilon,b-2\varepsilon]$
\end{center}
and $Sp(A)\subset\mathbb{U}$ for all $\mathcal{L}$. As a conclusion, property (a) of Subsection \ref{subsection5.1} holds.

\section{Convergence of solutions and non-resonance}

This last section is devoted to the convergence of $\mathbf{y}_\varepsilon(t)$ to $\mathbf{z}(t)$ in the case of multidimensional harmonic oscillator with $N=1$. 
For sake of conciseness, we suppose that 
\begin{itemize}
\item each lagrangian is real, stationary, with $J_k=0$, for all $k$, that is to say 
$\mathcal{L}(\mathbf{x},\dot{\mathbf{x}})=\frac 12P\dot{\mathbf{x}}^2+\frac12Q\mathbf{x}^2$, with $P$ and $Q$ real, constant and nonsingular,
\item the operator $\Box_\varepsilon$ is of the shape $\Box_\varepsilon^{[r,r]}$ with $r\in\mathbb{R}$. 
\end{itemize}
The setting is discussed at the end of the section. Even with those restrictions, the convergence is not unconditional w.r.t. $r$. 
\begin{lemma}
Let us assume that $\lim\limits_{\varepsilon\to 0}\mathbf{y}_\varepsilon (t)=\mathbf{z}(t)$ for all lagrangian such that the hypotheses of the Propositions \ref{prop1} and \ref{prop2} hold. Then, we have for some $r,s\in\mathbb{C}$
\begin{center}
$\Box_\varepsilon=\Box_\varepsilon^{[r,s]}$.
\end{center}
\label{lemma1}
\end{lemma}
\begin{proof}
Indeed, $\mathbf{y}_\varepsilon(t)$ and $\mathbf{z}(t)$ exist, are unique and pseudo-periodic on $[a,b]$. Let $\mathcal{L}_0$ be the lagrangian deduced from $\mathcal{L}$ by removing the terms $J_2$ and $J_3$ by $0$. Since the solutions $\mathbf{y}_\varepsilon$ and $\mathbf{y}_{\varepsilon,0}$ of D.E.L. associated to $\mathcal{L}$ and $\mathcal{L}_0$ tend to the solutions $\mathbf{z}$ and $\mathbf{z}_0$ of the respective C.E.L. we have $\mathbf{y}_\varepsilon-\mathbf{y}_{\varepsilon,0}=-Q^{-1}(\Box_{-\varepsilon}J_2+J_3)$ which tends to $\mathbf{z}-\mathbf{z}_{0}=-Q^{-1}J_3$. So, we obtain $\lim\limits_{\varepsilon\to 0}Q^{-1}J_2\times\Box_{-\varepsilon}1=0$ for all $J_2$. As a consequence, $\displaystyle \Box_{-\varepsilon}1=c_0+c_1+c_{-1}=0$. Thus, $\Box_\varepsilon=\Box_\varepsilon^{[r,s]}$ where $r=c_1$ and $s=-c_{-1}$.\end{proof}

\begin{lemma}
Let $\varphi(\varepsilon,\zeta)$ be a convergent Taylor series in some polydisc of $\mathbb{C}^2$ such that $\varphi(0,0)=0$. If $\delta>0$ and $B(0,\delta)\subset\mathbb{C}$ is the disc of radius $\delta$, let $g:B(0,\delta)\rightarrow\mathbb{C}$ be a continuous function such that $g(\varepsilon)\underset{0}{\sim}\frac{\tau}{\varepsilon}$ for some $\tau\in\mathbb{C}^\star$ and $|g(\varepsilon)-\frac{\tau}{\varepsilon}|$ is bounded. Then, for all entire function $\psi(\zeta)$ in $\mathbb{C}$, we have
\begin{equation}
\lim\limits_{\varepsilon\to 0}\psi(g(\varepsilon)\varphi(\varepsilon,\varepsilon\Omega))=\psi\left(\tau \frac{\partial\varphi}{\partial\varepsilon}(0,0)I_d+\tau \frac{\partial\varphi}{\partial \zeta}(0,0)\Omega\right).
\label{decadix}
\end{equation}
uniformly in any compact subset of $\mathbb{C}^{d\times d}$.
\label{lemma2}
\end{lemma}

\noindent\begin{proof}
Let $\|.\|$ be a norm of algebra over the Banach algebra $\mathbb{C}^{d\times d}$. We denote by $\varphi(\varepsilon,\zeta)=\sum_{i,j}a_{i,j}\varepsilon^i \zeta^j$ the Taylor series of $\varphi(\varepsilon,\zeta)$ in the bidisc  $|\varepsilon|+|\zeta|<r_0$ of $\mathbb{C}^2$. Let $B'(0,r_1)\subset\mathbb{C}^{d\times d}$ be an open ball of radius $r_1>0$. We set $\delta_1=\frac{1}{2}\min\left(\delta,\frac{r_0}{1+r_1}\right)$. The matrix-valued mapping  $(\varepsilon,\Omega)\mapsto\Theta:= \varphi(\varepsilon,\varepsilon\Omega)$ is well-defined and analytic in $B(0,\delta_1)\times B'(0,r_1)$. We have by assumption $a_{0,0}=0$ and we  denote $a_{1,0}=\frac{\partial\varphi}{\partial\varepsilon}(0,0)$ and $a_{0,1}=\frac{\partial\varphi}{\partial \zeta}(0,0)$. We get 
\begin{equation}
\|g(\varepsilon)\Theta-(\tau a_{1,0}I_d+\tau a_{0,1}\Omega)\|\leq |g(\varepsilon)-\frac{\tau}{\varepsilon}|\|\Theta\|+|\tau|\sum_{i+j\geq 2}|a_{i,j}||\varepsilon|^{i+j-1}\|\Omega\|^j.
\label{polydisc}
\end{equation}
Let $b_1>0$ such that $|g(\varepsilon)-\frac{\tau}{\varepsilon}|\leq b_1$ in $[-\delta,\delta]$. Let $\mathcal{K}$ be some compact subset of $B(0,\delta_1)\times B'(0,r_1)$. We use the fact that, in any bidisc $|\varepsilon|+|\zeta|<r_2<r_0$, the Taylor series $\sum_{i,j}|a_{ij}|\varepsilon^i \zeta^j$ is normally convergent. The function $\Theta/\varepsilon=\varphi(\varepsilon,\varepsilon\Omega)/\varepsilon$ being a Taylor series, we have $\|\Theta\|\leq b_2|\varepsilon|$ in $\mathcal{K}$ for some $b_2>0$. Similarly, the sum in (\ref{polydisc}) is upper bounded by $b_3|\varepsilon|$ in $\mathcal{K}$ for some $b_3>0$. Then, the l.h.s. of (\ref{polydisc}) is upper bounded by $\varepsilon(b_1b_2+|\tau|b_3)$ in $\mathcal{K}$. This implies that the convergence in (\ref{decadix}) holds and is uniform in the ball $B(0,r_1)$ provided we have $\psi(\zeta)=\zeta$. Now, if $\psi(\zeta)$ is an entire function, the l.h.s. (\ref{decadix}) gets sense and it is classical analysis that composition limit law holds for uniform convergence.
\end{proof}

\begin{theorem}
We assume that $M=(b-a)/\varepsilon\in\mathbb{N}^\star$ is even. Let us consider $\Box_\varepsilon=\Box_\varepsilon^{[r,r]}$, for some $r\in\mathbb{R}^\star$. Let $\mathcal{L}$ be a real  stationary quadratic lagrangian of the shape $\mathcal{L}(\mathbf{x},\dot{\mathbf{x}})=\frac 12P\dot{\mathbf{x}}^2+\frac12Q\mathbf{x}^2$, with $P$ and $Q$ real, constant and nonsingular. We suppose that $-P^{-1}Q=\Omega^2$ and $Sp(\Omega)\subset\mathbb{U}$ for some matrix $\Omega\in\mathbb{R}^{d\times d}$ diagonalizable over $\mathbb{R}$. We require also the non-resonance property
\begin{equation}
\frac{1}{\pi}(b-a)\omega\notin \mathbb{Q}~\mbox{  and  }~\frac 1\pi \arccos \left(1-\frac{(b-a)^2\omega^2}{2r^2n^2}\right) \notin \mathbb{Q},\forall n\in\mathbb{N}^\star.
\label{resonance}
\end{equation}
for all eigenvalue $\exp(i\omega)$ in $Sp(\Omega)\cap\mathbb{U}$. Then $\mathbf{y}_\varepsilon(t)$ is well-defined. Moreover, $r=\pm\frac{1}{2}$ if and only if for all $\mathbf{d}_a,\mathbf{d}_b\in\mathbb{C}^d$, $\mathbf{y}_\varepsilon (t)$ tends to $\mathbf{z}(t)$ uniformly on $[a,b]$ as $\varepsilon\rightarrow 0$.
\label{maintheorem}
\end{theorem}

\noindent\begin{proof}
The necessary conditions of the first order C.E.L. and D.E.L., that is (\ref{elho}) and (\ref{mchomat}), simplify into 
\begin{equation}
-P\ddot {\mathbf{z}}(t)+Q\mathbf{z}(t)=0,\hspace{1cm}P\Box _{-\varepsilon
}\Box _\varepsilon \mathbf{y}_\varepsilon (t)+Q\mathbf{y}_\varepsilon (t)=0,
\label{oscillunidim}
\end{equation}
completed with Dirichlet conditions. Let $\mathbf{z}(t)$ be the solution of (\ref{oscillunidim}), that is to say
\begin{center}
$\mathbf{z}(t)=\exp(i\Omega(t-a))\mathbf{f}_1+\exp(-i\Omega (t-a))\mathbf{f}_2$,
\end{center}
where $t\in [a,b]$ and $\displaystyle\Omega^2=-P^{-1}Q$. Due to (\ref{resonance}), the diagonalizable matrix $\sin((b-a)\Omega)$ is invertible and we may solve the boundary conditions $\mathbf{z}(a)=\mathbf{d}_a$ and $\mathbf{z}(b)=\mathbf{d}_b$. If we set 
\begin{equation}
\mathbf{F}(\tau,\mathbf{d}_a,\mathbf{d}_b)=\frac{i}{2}(\sin (\tau\Omega)^{-1}(\exp(-i\tau\Omega)\mathbf{d}_a-\mathbf{d}_b)
\label{Fgras}
\end{equation}
we get $\mathbf{f}_1=\mathbf{F}(b-a,\mathbf{d}_a,\mathbf{d}_b)$ and $\displaystyle \mathbf{f}_2=\mathbf{F}(a-b,\mathbf{d}_a,\mathbf{d}_b)$.

Since we have $c_1=r/\varepsilon$, $c_{-1}=-r/\varepsilon$ and $c_0=0$, then, for all $t\in[a+2\varepsilon,b-2\varepsilon]$, D.E.L. in  (\ref{oscillunidim}) may be simplified into 
\begin{equation}
\mathbf{x}(t+2\varepsilon)+\mathbf{x}(t-2\varepsilon)=2\left(I_d+\frac{\varepsilon^2}{2r^2}P^{-1}Q\right)\mathbf{x}(t).
\label{mchomatoh}
\end{equation}
Since $\Omega$ is diagonalizable over $\mathbb{R}$, for some $B\in\mathbb{R}^{d\times d}$, we have $B^{-1}\Omega B=\mbox{diag}(\omega_i)$. For $\varepsilon>0$ such that $\varepsilon<2|r|\min|\omega_i|$, the matrix
\begin{equation} \Theta=B\mbox{diag}\left(\arcsin\left(\frac{\varepsilon}{r}\omega_i\sqrt{1-\frac{\varepsilon^2}{4r^2}\omega_i^2}\right)\right)B^{-1}\in\mathbb{R}^{d\times d}
\label{theta}
\end{equation}
is well-defined and the computation of $\cos(B^{-1}\Theta B)$ gives $\cos(\Theta)=Id-\frac{\varepsilon^2}{2r^2}\Omega^2$. By setting $\mathbf{u}_{n}=\mathbf{x}_\varepsilon(a+2n\varepsilon)$, the equation (\ref{mchomatoh}) gives $\mathbf{u}_{n+1}=2\cos(\Theta)\mathbf{u}_n-\mathbf{u}_{n-1}$. The solution of this recurrence is given by
\begin{center}
$\mathbf{x}_\varepsilon (a+2n\varepsilon)=\exp(in\Theta)\mathbf{g}_1+\exp(-in\Theta)\mathbf{g}_2$
\end{center}
for all $n$ from 1 to $\frac{1}{2}M-1$.\\
To compute the vectors $\mathbf{g}_1$ and $\mathbf{g}_2$ we use the Dirichlet conditions for $\mathbf{y}_\varepsilon$ and the invertibility of the matrices $\sin(k\Theta)$, $k\geq1$. Indeed, it is ensured since we have 
\begin{center}
$\displaystyle \det(\sin(k\Theta))=\prod_{i=1}^d\sin\left(k\arccos\left(1-\frac{(b-a)^2\omega^2}{2r^2M^2}\right)\right)\neq 0$ for all $k,M\in\mathbb{N}^\star$
\end{center} 
help to (\ref{resonance}). Now, the equation (\ref{mchomat3}) gives for $t=a$ and $t=b$ the boundary conditions 
\begin{center}
$\displaystyle\mathbf{x}_\varepsilon (a+2\varepsilon)=\mathbf{d}_a^{\prime}=\left(I_d-\frac{\varepsilon^2}{r^2}\Omega^2\right)\mathbf{d}_a$ and $\displaystyle \mathbf{x}_\varepsilon
(b-2\varepsilon)=\mathbf{d}_b^{\prime}=\left(I_d-\frac{\varepsilon^2}{r^2}\Omega^2\right)\mathbf{d}_b$.
\end{center}
Note that $\mathbf{d}_a^{\prime}$ and $\mathbf{d}_b^{\prime}$ tend to $\mathbf{d}_a$ and $\mathbf{d}_b$
respectively as $\varepsilon$ tends to 0. Next, (\ref{mchomat3}) gives for $t=a+\varepsilon$ and $t=b-\varepsilon$ the vectors
\begin{center}
$\begin{array}{l}
\mathbf{g}_1=\frac{i}{2}\left(\sin\left(\left(\frac{1}{2}M-2\right)\Theta\right)\right)^{-1}\left(\exp\left(-i\left(\frac{1}{2}M-1\right)\Theta\right)\mathbf{d}_a^{\prime}-\exp(-i\Theta)\mathbf{d}_b^{\prime}\right)\vspace{0.1cm}\\
\mathbf{g}_2=\frac{i}{2}\left(\sin\left(\left(\frac{1}{2}M-2\right)\Theta\right)\right)^{-1}\left(\exp(i\Theta)\mathbf{d}_b^{\prime}-\exp\left(i\left(\frac{1}{2}M-1\right)\Theta\right)\mathbf{d}_a^{\prime}\right).
\end{array}$\vspace{0.1cm}
\end{center}
We also need in the following
\begin{center}
$\begin{array}{l} \tilde{\mathbf{g}}_1=\frac{i}{2}\left(\sin\left(\frac{1}{2}M\Theta\right)\right)^{-1}\left(\exp\left(-i\frac{1}{2}M\Theta\right)\mathbf{d}_a-\mathbf{d}_b\right)\vspace{0.1cm}\\ \tilde{\mathbf{g}}_2=\frac{i}{2}\left(\sin\left(\frac{1}{2}M\Theta\right)\right)^{-1}\left(\mathbf{d}_b-\exp\left(i\frac{1}{2}M\Theta\right)\mathbf{d}_a\right).
\end{array}$
\end{center}
\vspace{0.1cm}
Since the sequence $\mathbf{u}_n$ is well-determined, Proposition \ref{prop3} and its proof show how $\mathbf{x}_\varepsilon(t)$ may be extended to an unique continuous pseudo-periodic function $\mathbf{y}_\varepsilon (t)$ over $[a,b]$.\\
We will apply several times Lemma \ref{lemma2} with $\varphi(\varepsilon,\zeta)=\arcsin\left(\frac{1}{r}\zeta\sqrt{1-\frac{1}{4r^2}\zeta^2}\right)$, see (\ref{theta}). Keeping notation of the proof of Lemma \ref{lemma2}, we have  $a_{1,0}=0$ and $a_{0,1}=\frac{1}{r}$. If  $g(\varepsilon)=\frac{b-a}{2\varepsilon}$, we get $\psi(g(\varepsilon)\Theta)\underset{\varepsilon\rightarrow 0}{\rightarrow} \psi(\frac{b-a}{2r}\Omega)$. Applying this result to $\psi(\zeta)=\cos(\zeta)$ or $\sin(\zeta)$ or $\exp(\pm i\zeta)$ we readily obtain $\displaystyle\lim\limits_{\varepsilon\to 0}\mathbf{g}_j=\lim\limits_{\varepsilon\to 0}\tilde{\mathbf{g}}_j=\mathbf{f}_j'$, $j=1,2$, where
\begin{equation}
\mathbf{f}_1'=\mathbf{F}\left(\frac{b-a}{2r},\mathbf{d}_a,\mathbf{d}_b\right)\mbox{ and }\mathbf{f}_2'=\mathbf{F}\left(\frac{a-b}{2r},\mathbf{d}_a,\mathbf{d}_b\right).
\label{fprime}
\end{equation}

Now, all these preliminaries being done, let us discuss the convergence of $\mathbf{y}_\varepsilon$ to $\mathbf{z}$. We fix $\delta>0$. If $t\in[a+\delta,b-\delta]$, we choose $g(\varepsilon)=\left[\frac{t-a}{2\varepsilon}\right]$ to represent the integer $n$ such that $t=a+2n\varepsilon$. Hence, with $\tau=t-a$, $g(\varepsilon)\sim \frac{\tau}{2\varepsilon}$ and $|g(\varepsilon)-\frac{\tau}{2\varepsilon}|\leq 1$. If we define $\mathbf{Z}$ by
\begin{equation}
\mathbf{Z}(\tau_1,\tau_2,\mathbf{d}_a,\mathbf{d}_b)=(\sin(\tau_2\Omega))^{-1}\sin(\tau_1\Omega)\mathbf{d}_b-(\sin(\tau_2\Omega))^{-1}\sin((\tau_1-\tau_2)\Omega)\mathbf{d}_a
\label{Zgras}
\end{equation}
for suitable numbers $\tau_1$ and $\tau_2$, then we have $\displaystyle \mathbf{z}(t)=\mathbf{Z}(t-a,b-a,\mathbf{d}_a,\mathbf{d}_b)$ and we are going to prove that $\mathbf{y}_\varepsilon(t)$ tends uniformly locally on $]a,b[$ to 
\begin{equation}
\mathbf{z}'(t):=\mathbf{Z}\left(\frac{t-a}{2r},\frac{b-a}{2r},\mathbf{d}_a,\mathbf{d}_b\right).
\label{zlimite}
\end{equation}
Let us note that $\mathbf{z}'(t)$ does not stand for the derivative that we have denoted $\dot{\mathbf{z}}(t)$. The vector $\mathbf{z}'(t)-\mathbf{y}_\varepsilon(t)$ may be written as
\begin{equation*}
\exp(i\tau\Omega)\mathbf{f}_1'+\exp(-i\tau\Omega)\mathbf{f}_2'-(\exp(in\Theta)\mathbf{g}_1+\exp(-in\Theta)\mathbf{g}_2).
\end{equation*}
Let us introduce the quantities
\begin{center}
$\begin{array}{rcl} 
\mathbf{q}_1 & = & \displaystyle -\exp(in\Theta)(\mathbf{g}_1-\tilde{\mathbf{g}}_1)\vspace{0.1cm}\\
\mathbf{q}_2 & = & \displaystyle \exp(-in\Theta)(\mathbf{g}_2-\tilde{\mathbf{g}}_2)\vspace{0.1cm}\\
\mathbf{q}_3 & = & \displaystyle \left(\cos\left(\frac{\tau}{2r}\Omega\right)-\cos(n\Theta)\right)\mathbf{d}_a\vspace{0.1cm}\\
\mathbf{q}_4 & = & \displaystyle i\sin\left(\frac{\tau}{2r}\Omega\right)(\mathbf{f}_1'-\mathbf{f}_2')-i\sin (n\Theta)(\tilde{\mathbf{g}}_1-\tilde{\mathbf{g}}_2).\vspace{0.1cm}
\end{array}$
\end{center} 
\vspace{0.1cm}
Straightforward computations yield $\mathbf{z}'(t)-\mathbf{y}_\varepsilon(t)=\mathbf{q}_1+\ldots+\mathbf{q}_4$, as a consequence of the two equalities 
$\mathbf{f}_1'+\mathbf{f}_2'=\tilde{\mathbf{g}}_1+\tilde{\mathbf{g}}_2=\mathbf{d}_a$.
Let $\|v\|$ be any norm on $\mathbb{C}^d$. Let us prove that $\lim\limits_{\varepsilon\to 0}\mathbf{q}_i=\mathbf{0}$ uniformly in $[a,b]$, for $i=1,\ldots,4$. As we have seen in the discussion before (\ref{fprime}), the vectors $\tilde{\mathbf{g}}_i-\mathbf{g}_i$ tend to 0 and $\exp(\pm in\Theta)$ is bounded since $\Theta$ is real so the vectors $\mathbf{q}_1$ and $\mathbf{q}_2$ tend to 0 uniformly on $[a,b]$. The case of $\mathbf{q}_3$ is obvious by using Lemma \ref{lemma2}, while $\mathbf{q}_4$ is more complicated. We obtain that $\|\mathbf{q}_4\|$ is less than 
\begin{center}
$\left\|\left(\sin\left(\frac{b-a}{2r}\Omega\right)\right)^{-1}\cos\left(\frac{b-a}{2r}\Omega\right)\sin\left(\frac{\tau}{2r}\Omega\right)%
-\left(\sin\left(\frac{1}{2}M\Theta\right)\right)^{-1}\cos\left(\frac{1}{2}M\Theta\right)\sin(n\Theta)\right\|.\|\mathbf{d}_a \|+\left\|\left(\sin\left(\frac{b-a}{2r}\Omega\right)\right)^{-1}\sin\left(\frac{\tau}{2r}\Omega\right)-\left(\sin\left(\frac{1}{2}M\Theta\right)\right)^{-1}\sin(n\Theta)\right\|.\|\mathbf{d}_b\|$
\end{center}
that is 
\begin{equation}
\displaystyle\left\|\mbox{cotg}\left(\frac{b-a}{2r}\Omega\right)-\mbox{cotg}\left(\frac{1}{2}M\Theta\right)\right\|.\|\mathbf{d}_a\|+\left\|\mbox{cosec}\left(\frac{b-a}{2r}\Omega\right)-\mbox{cosec}\left(\frac{1}{2}M\Theta\right)\right\|.\left\|\mathbf{d}_b\right\|.
\label{majo}
\end{equation}
We have $\|\Theta\|\leq\|B\|\|B^{-1}\||\varepsilon|\frac{1}{r}\max(|\omega_i|(1-\frac{\varepsilon^2}{4r^2}\omega_i^2))$ where $B$ diagonalizes $\Omega$. So, Lemma (\ref{lemma2}) with $\psi(\zeta)=\cos(\zeta)$ and next $\psi(\zeta)=\sin(\zeta)$ shows that the bound (\ref{majo}) tends to 0 for all $t\in[a,b]$. But this convergence is also uniform in $[a,b]$ due to formula (\ref{polydisc}), to the previous bound of $\Theta$ and to the boundedness of $g(\varepsilon)-\frac{\tau}{2\varepsilon}$. Hence, by using notation in (\ref{zlimite}), we have proved so far that
\begin{center}
$\mathbf{y}_\varepsilon(t)\underset{\varepsilon\rightarrow 0}{\rightarrow}\mathbf{z}'(t)$ uniformly in $[a,b]$.
\end{center}
Let us show that, if $r\neq\pm\frac{1}{2}$, then we can choose $t\in[a,b]$ and the vectors $\mathbf{d}_a,\mathbf{d}_b\in\mathbb{C}^d$ in such a way that $\mathbf{z}'(t)\neq \mathbf{z}(t)$. Indeed, inspection of (\ref{Zgras}) shows that the coefficient of $\mathbf{d}_a$ is the matrix 
\begin{equation}
\left(\sin\left(\frac{b-a}{2r}\Omega\right)\right)^{-1}\sin\left(\frac{t-b}{2r}\Omega\right)-\left(\sin((b-a)\Omega)\right)^{-1}\sin((t-b)\Omega),
\label{prodcroix}
\end{equation}
where $\sin\left(\frac{b-a}{2r}\Omega\right)$ and $\sin((b-a)\Omega)$ are invertible due to (\ref{resonance}). We choose $t$ such that for all $\omega\in Sp(\Omega)$ we have 
\begin{center}
$\left(\sin\left(\frac{b-a}{2r}\omega\right)\right)^{-1}\sin\left(\frac{t-b}{2r}\omega\right)-(\sin((b-a)\omega))^{-1}\sin((t-b)\omega)\neq 0$.
\end{center}
In this way, given any vector $\mathbf{d}_b$ we may choose $\mathbf{d}_a$ so that $\mathbf{z}'(t)-\mathbf{z}(t)=e_1$ where $e_1$ is the first vector of the canonical basis of $\mathbb{C}^{d}$ since the matrix (\ref{prodcroix}) is invertible.\\
If $r=\pm 1/2$, we have $\mathbf{z}'(t)=\mathbf{z}(t)$ and the proof is complete.\end{proof} 

\begin{remark}\rm
If $M=(b-a)/\varepsilon$ is odd, the system (\ref{solfin}) for determining $\mathbf{d}_s=\mathbf{x}_\varepsilon(a+\varepsilon)$ from $\mathbf{d}_a$ and $\mathbf{d}_b$ is not Cramer. Indeed, D.E.L. in (\ref{oscillunidim}) simplifies into (\ref{mchomatoh}). Since this recurrence does not match $a$, $a+\varepsilon$ and $b$, the matrix occuring in (\ref{cramer}) is the zero matrix $O_d$. So, when $\varepsilon$ is not of the shape alluded in the previous theorem, the convergence is not guaranteed (see the numerical experiments below).
\end{remark}

\begin{remark}\rm
The assumption $J_k=0$, $k\geq 2$, is not restrictive. Let $\mathcal{L}$ be defined as in (\ref{bilineaire}) and $\mathcal{L}_0(\mathbf{x},\dot{\mathbf{x}})=\frac 12P\dot{\mathbf{x}}^2+\frac12Q\mathbf{x}^2$. Let $\mathbf{y}_{\varepsilon}$, $\mathbf{y}_{\varepsilon,0}$, $\mathbf{z}$ and $\mathbf{z}_0$ be the solutions of D.E.L. and C.E.L. for $\mathcal{L}$ and $\mathcal{L}_0$. We shall see that the property of convergence holds for $\mathcal{L}$ iff it holds for $\mathcal{L}_0$. Indeed, we have the formulas $\mathbf{y}_\varepsilon=\mathbf{y}_{\varepsilon,0}-Q^{-1}(\Box_{-\varepsilon}J_2+J_3)$, $\mathbf{z}=\mathbf{z}_{0}-Q^{-1}J_3$ and $\lim\limits_{\varepsilon\to 0}\mathbf{y}_{\varepsilon,0}=\mathbf{z}_0$. So we have $\lim\limits_{\varepsilon\to 0}\mathbf{y}_{\varepsilon}=\mathbf{z}\Leftrightarrow \lim\limits_{\varepsilon\to 0}\mathbf{y}_{\varepsilon,0}=\mathbf{z}_0$.
\end{remark}

\begin{remark}\rm
We already obtained in \cite{RS2} two characterizations of $\Box_\varepsilon^{[r,s]}$ among all operators $\Box_\varepsilon$. The first one was linked to the convergence of the l.h.s. of (\ref{mchomat}) to the l.h.s. of (\ref{elho}) for all lagrangian (\ref{bilineaire}) and led to the relation $r+s=1$. The second one ensured that $Sp(A)\subset\mathbb{U}$ if $d=1$, which leads to $s=\overline{r}$. Now, we may add a third family which consists in operators $\Box_\varepsilon^{[r,s]}$, for which the five-terms recurrence (\ref{mchomat}) splits into two three-terms recurrences, one for $\mathbf{x}_{2n}$ and $\mathbf{x}_{2n+1}$, this being equivalent to $r=s$.
\end{remark}

\section{Numerical experiments}

Let us illustrate the phenomenon of convergence proved in Theorem \ref{maintheorem}, when $M$ increases. We set in every example below $\mathbf{d}_a=12$, $\mathbf{d}_b =-14$, $p=1$ and $q=-0.23$

Figures \ref{fig1-a.b} below illustrate the behaviour of $\mathbf{y}_\varepsilon(t)$ and $\mathbf{z}(t)$ when $M$ increases ($M=30$ and $M=120$). We choose first  $\gamma_1=1/2$, $a=0$ and $b=30$, so condition (\ref{resonance}) is true.
\begin{figure}[th]
\includegraphics[width=4.6cm,angle=270]{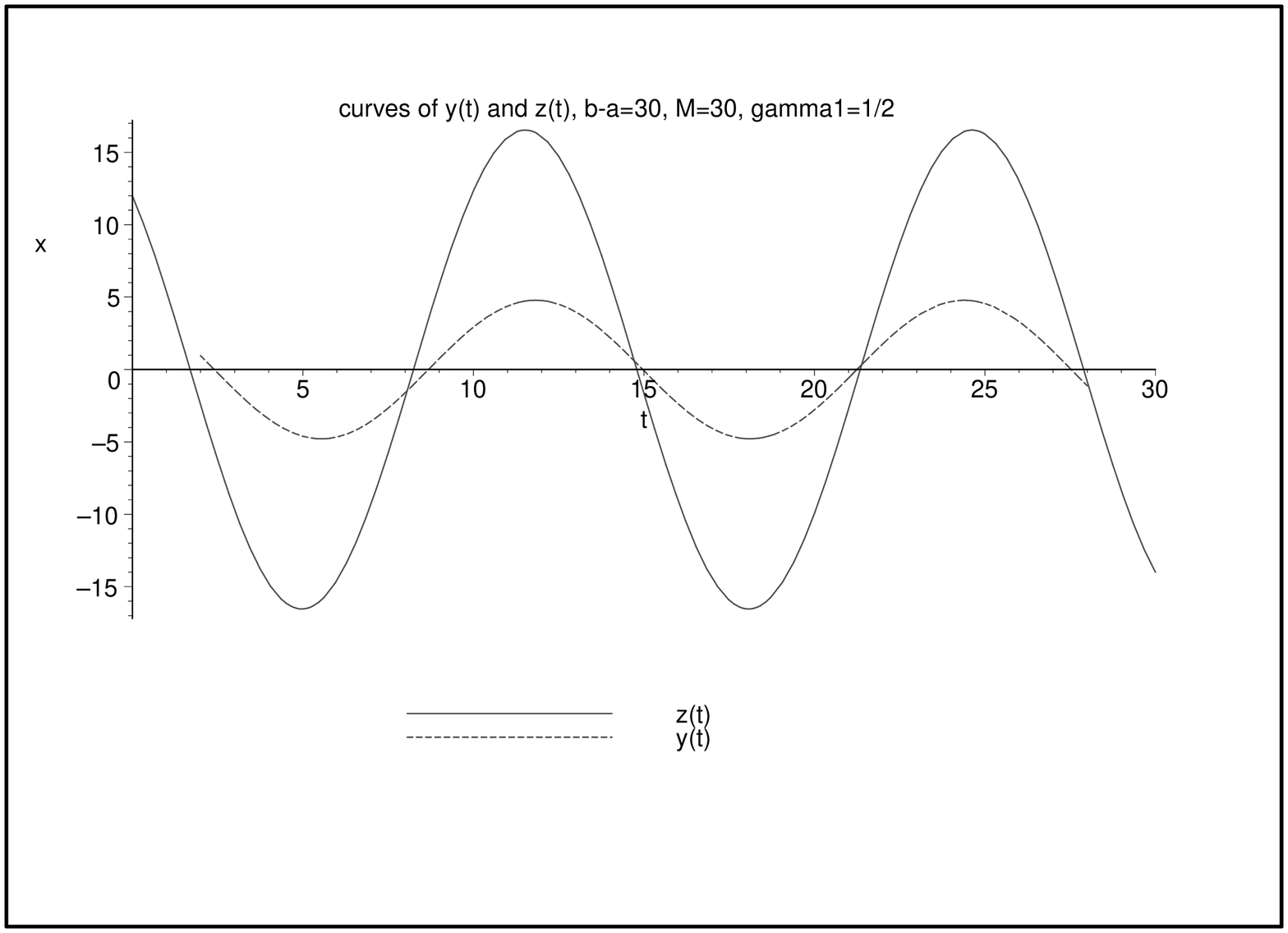}~\includegraphics[width=4.6cm,angle=270]{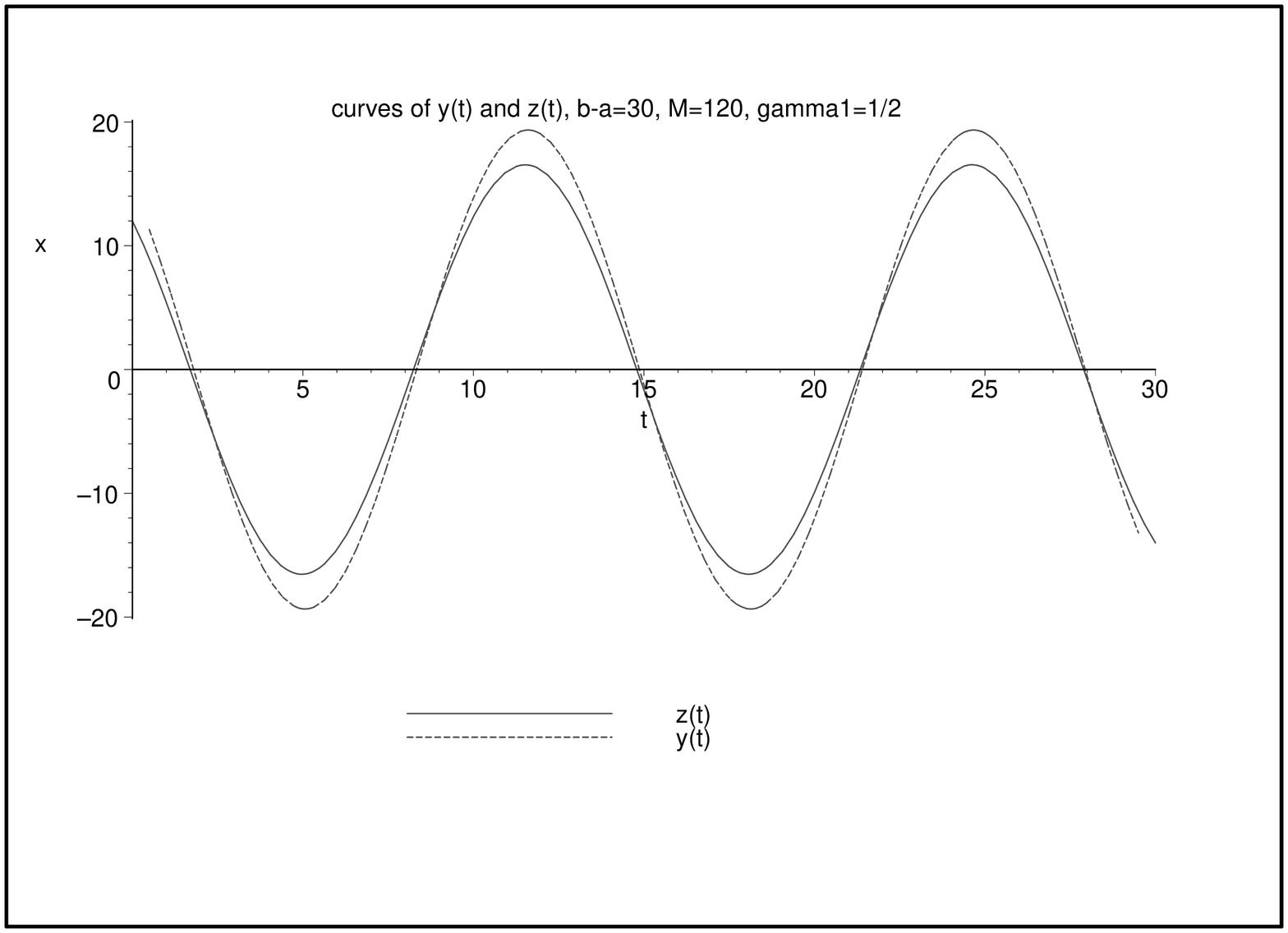}
\caption{\textit{Convergence in non-resonant case ($M=30$ and $M=120$)}}
\label{fig1-a.b}
\end{figure}

As soon as the condition (\ref{resonance}) for the continuous lagrangian fails, the convergence does not occur. However, if $b=a+\frac 1\omega (\arcsin(\rho )+2K\pi)$, for $\rho$ tending to 0, the upper bound (\ref{majo}) grows to infinity. The small denominators~ $\sin(M\Theta)$ and $\sin(\tau \omega)$ occuring in $\mathbf{y}_\varepsilon(t)$ and $\mathbf{z}(t)$ imply that the convergence holds but is slowed down (see Figures \ref{fig2-a.b}).
\begin{figure}[th]
\includegraphics[width=4.6cm,angle=270]{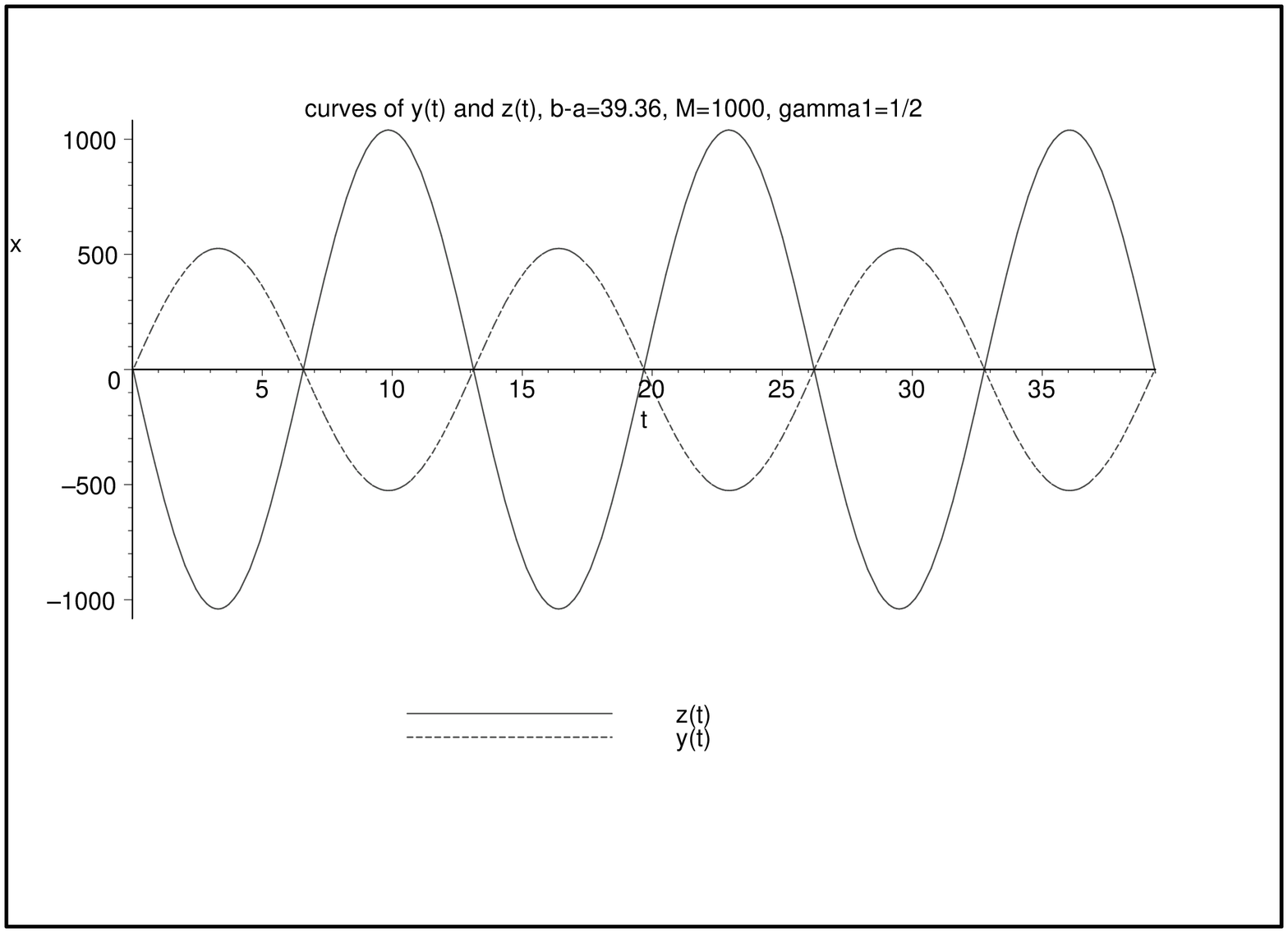}~\includegraphics[width=4.6cm,angle=270]{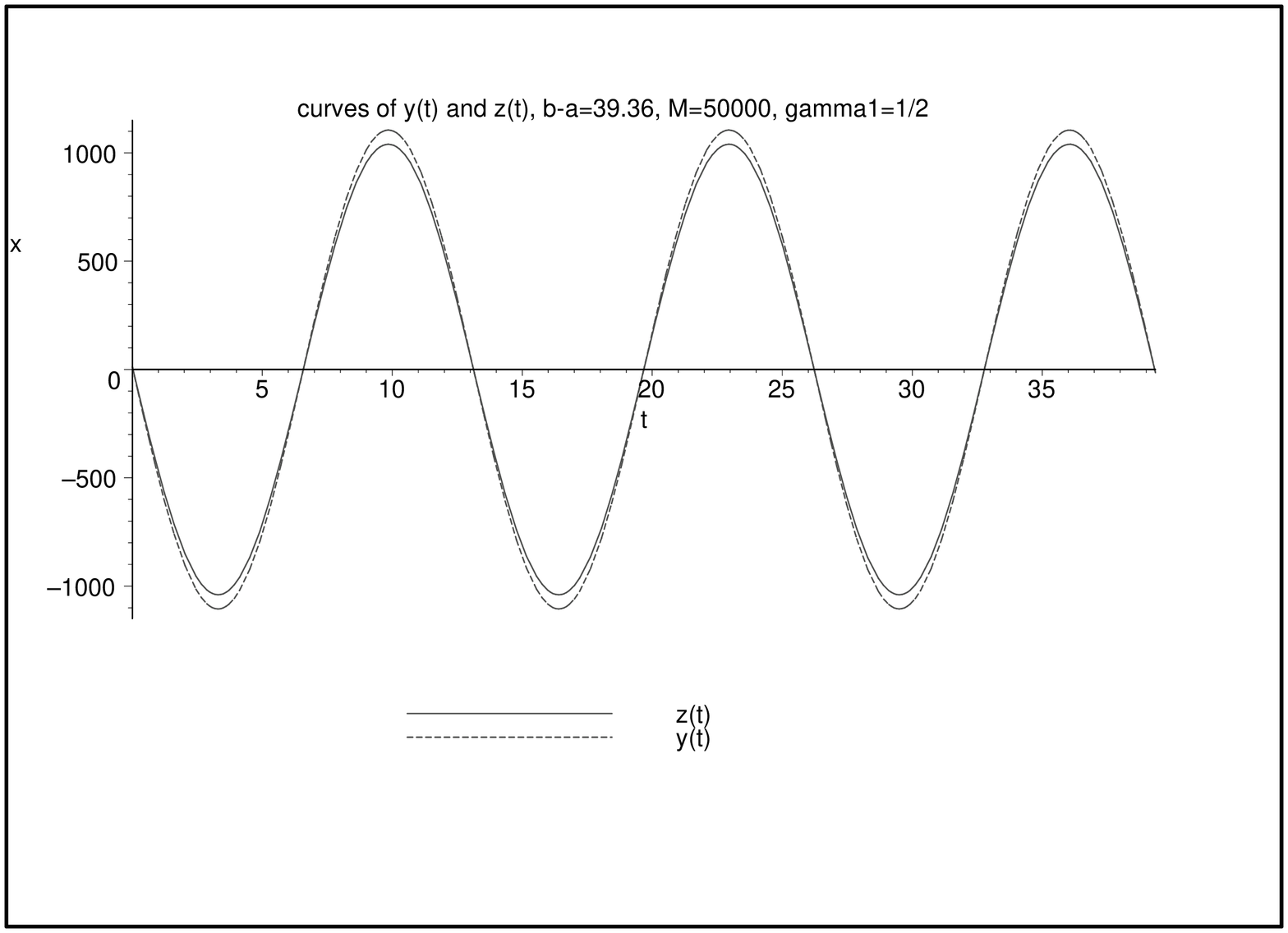}
\caption{\textit{Slow convergence in quasi-resonant case ($M=1000$ and $M=50000$)}}
\label{fig2-a.b}
\end{figure}

Next, let us present two examples of non-convergence of solutions, that is to say when $r=\gamma_1\neq \pm\frac{1}{2}$. We remind that the study of the convergence of the operators in D.E.L. to the operators in C.E.L. is studied in \cite[pp.7, Theorem 6.3]{RS2}, and may be shortened as : $\Box=\Box_\varepsilon^{[r,s]}$. Next, the existence of pseudo-periodic solutions implies $r+s=1$ (see \cite[pp.7, Proposition 5.2]{RS2}). If $r=s=\gamma_1\neq\frac{1}{2}$, then $\Box_\varepsilon^{[r,r]}$ does not fullfil the previous requirements. We display in Figures \ref{fig3-a.b} two such instances with $\gamma_1=0.6$ and $\gamma_1=(1+i)/2$.

\begin{figure}[ht]
\includegraphics[width=4.6cm,angle=270]{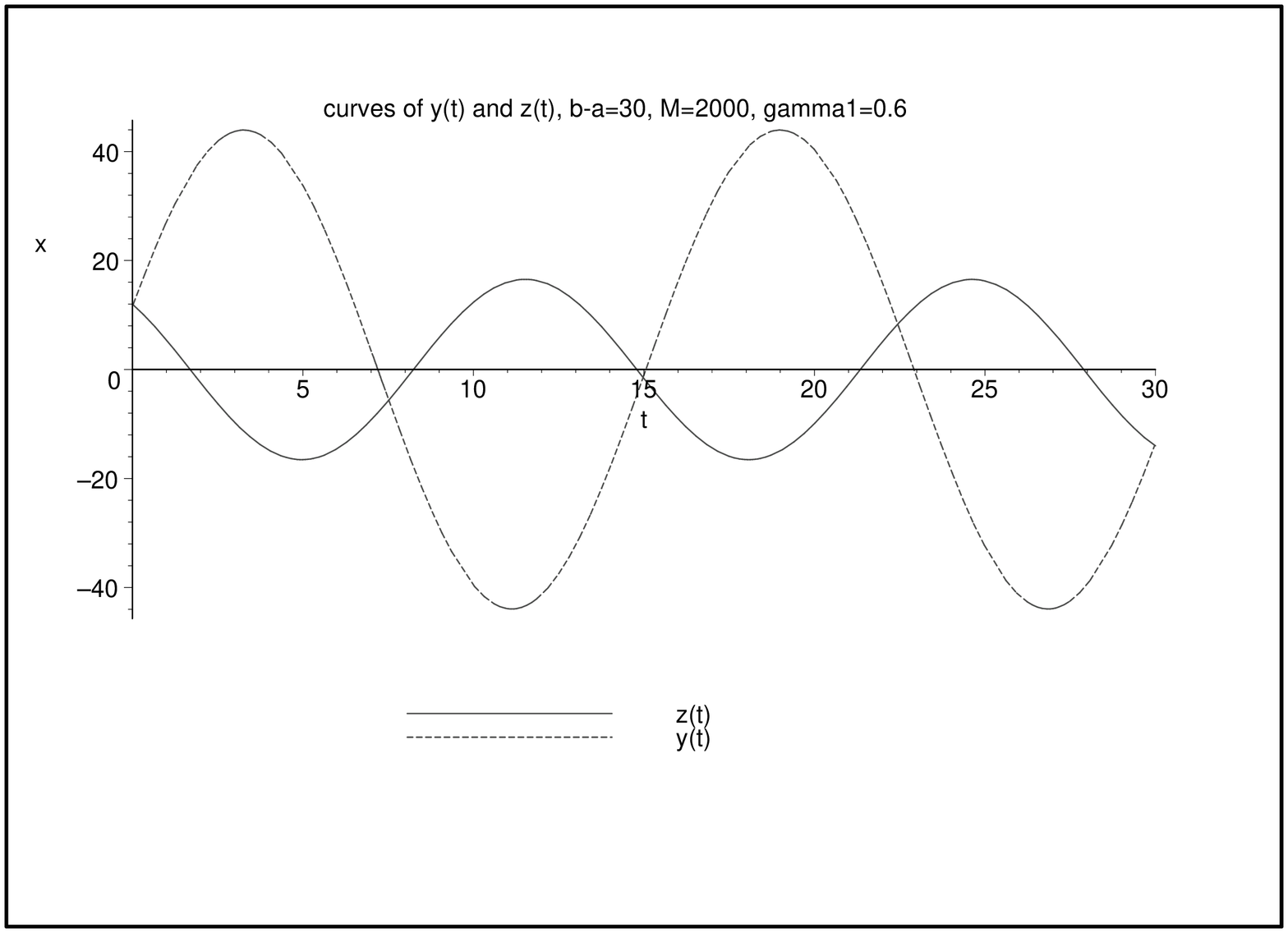}~\includegraphics[width=4.6cm,angle=270]{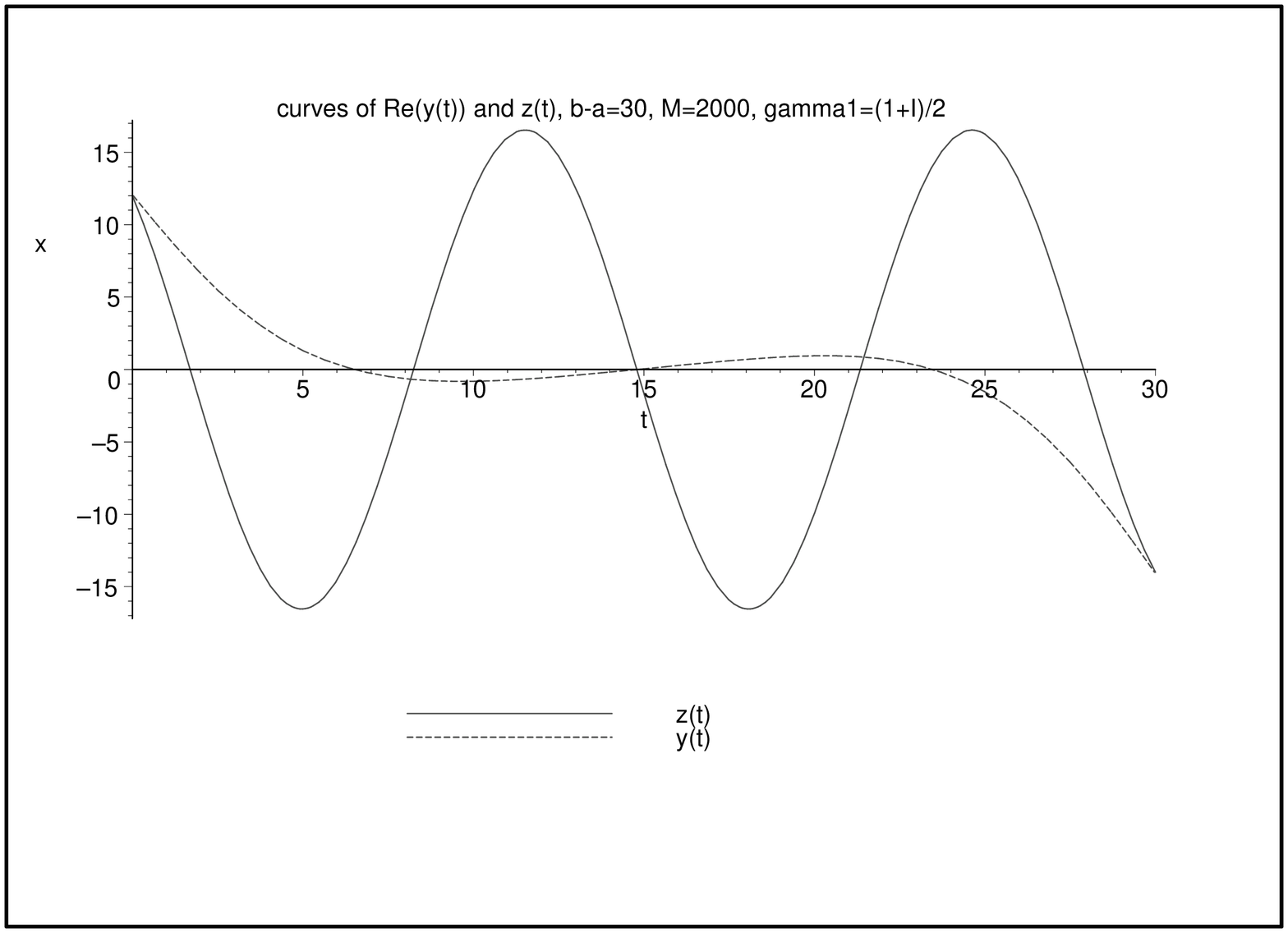}
\caption{\textit{Non-convergence phenomenon ($\gamma_1=0.6$ and $\gamma_1=(1+i)/2$)}}
\label{fig3-a.b}
\end{figure}

Let us conclude this paper with the following problems. First, formal and numerical codes have been written to do experiments on D.E.L. and C.E.L. in higher dimension ($d\geq 2$), to deal with huge matrices $A$ and to work with non-periodic solutions. However, mainly due to the characteristic functions $\chi_i(t)$ occuring in $\Box_\varepsilon$, no general pattern has been found neither for the convergence nor for non-convergence. Second, it would be interesting to get qualitative properties as continuity or mesurability of solutions $\mathbf{x}_\varepsilon(t)$ of D.E.L. as it is usual in the theory of functional equations. Another work is to relate the convergence of schemes to the convergence of solutions, none of these properties implying the other. These directions seem to be some interesting perspectives for subsequent work.


\end{document}